\documentclass[aos, reqno, preprint, 10pt]{article}%

\RequirePackage{amsthm, amsmath, natbib, amsfonts, amssymb}%

\usepackage[latin1]{inputenc}
\usepackage{graphicx, color}%
\usepackage{tikz}%
\usepackage{natbib}%
\usepackage{mathrsfs}

\usepackage{tikz}

\definecolor{darkblue}{rgb}{0.0,0.0,0.7}

\RequirePackage[%
colorlinks = true,%
linkcolor = darkblue,%
citecolor = darkblue,%
urlcolor = darkblue, %
]{hyperref}%

\hypersetup{%
  pdfauthor = {St\'ephane Ga\"iffas, Agathe Guilloux},%
  pdftitle = {Learning for marker-dependent counting processes},%
  pdfcreator = {pdflatex},%
  pdfproducer = {pdflatex}}

\newcommand{\T}{^{\top}}%
\newcommand{\prodsca}[2]{\langle #1,#2 \rangle}%
\newcommand{\norm}[1]{\|#1\|}%
\newcommand{\ind}[1]{\mathbf 1_{#1}}%

\newcommand{\E}{\mathsf E}
\renewcommand{\L}{\mathbb L}
\renewcommand{\P}{\mathsf P}
\renewcommand{\d}{{\rm d}}

\DeclareMathOperator{\Span}{span}

\newcommand{\1}{{\rm 1}\kern-0.24em{\rm I}}

\newtheorem{theorem}{Theorem}%
\newtheorem{corollary}{Corollary}%
\newtheorem{lemma}{Lemma}%
\newtheorem{proposition}{Proposition}%
\theoremstyle{assumption}%
\newtheorem{assumption}{Assumption}{\bf}{\rm}%
\theoremstyle{remark}%
\newtheorem{remark}{Remark}%
\newtheorem{definition}{Definition}%

\newcommand{\prodi}{{\textrm {\Huge $\pi$}}}

\newcommand{\bs}{\boldsymbol}

\begin{document}

\title{Learning and adaptive estimation for marker-dependent counting
  processes}

\author{St\'ephane Ga\"iffas${}^{1}$ \and Agathe Guilloux${}^{1,2}$}

\footnotetext[1]{Universit\'e Pierre et Marie Curie - Paris~6,
  Laboratoire de Statistique Th\'eorique et Appliqu\'ee, 175 rue du
  Chevaleret, 75013 PARIS. Supported in part by ANR Grant
  \textsc{``Prognostic''} \\ \emph{Email}:
  \texttt{stephane.gaiffas@upmc.fr}}

\footnotetext[2]{Universit\'e Pierre et Marie Curie - Paris~6, Unité
  INSERM 762 "Instabilité des Microsatellites et Cancers"  \\
  \emph{Email}: \texttt{agathe.guilloux@upmc.fr}}

\maketitle

\begin{abstract}
  We consider the problem of statistical learning for the intensity of
  a counting process with covariates. In this context, we introduce an
  empirical risk, and prove risk bounds for the corresponding
  empirical risk minimizers. Then, we give an oracle inequality for
  the popular algorithm of aggregation with exponential weights. This
  provides a way of constructing estimators that are adaptive to the
  smoothness and to the structure of the intensity. We prove that
  these estimators are adaptive over anisotropic Besov balls. The
  probabilistic tools are maximal inequalities using the generic
  chaining mechanism, which was introduced by \cite{talagrand06},
  together with Bernstein's
  inequality for the underlying martingales. \\
  
  \noindent%
  \emph{Keywords.} Counting processes, Statistical learning, Adaptive
  estimation, Empirical risk minimization, Aggregation with
  exponential weights, Generic chaining
\end{abstract}

\section{Introduction}
\label{sec:introduction}

Over the last decade, statistical learning theory (initiated by
Vapnik, see for instance \cite{vapnik00}) has known a tremendous
amount of mathematical developments. By mathematical developments, we
mean risk bounds for learning algorithms, such as empirical risk
minimization, penalization or aggregation. However, in the vast
majority of papers, such bounds are derived in the context of
regression, density or classification. In the regression model, one
observes independent copies of $(X, Y)$, where $X$ is an input, or a
covariate, and $Y$ is a real output, or label. The aim is then to
infer on $\E(Y | X)$. The aim of this paper is to study the same
learning algorithms (such as empirical risk minimization) in a more
sophisticated setting, where the output is not a real number, but a
stochastic process. Namely, we focus on the situation where, roughly,
the output is a counting process, which has an intensity that depends
on the covariate $X$. The aim is then to infer on this intensity. This
framework contains many models, that are of importance in practical
situations, such as in medicine, actuarial science or econometrics,
see~\cite{ABGK}.

In this paper, we give risk bounds for empirical risk minimization and
aggregation algorithms. In summary, we try to ``find back'' the kind
of results one usually has in more ``standard'' models (see below for
references). Then, as an application of these results, we construct
estimators that have the property to adapt to the smoothness and to
the structure of the intensity (in the context of a single-index
model). Several papers work in a setting close to ours. Model
selection has been first studied in \cite{Patricia1} for the
non-conditional intensity of a Poisson process, see also
\cite{Patricia2}, \cite{birge-2007-55}, \cite{MR2449129} and
\cite{brunel_comte05}. Model selection for the same problem as the one
considered here has been studied in \cite{CGG}.

The agenda of the paper is the following. In this Section, we describe
the general setting and the corresponding estimation
problem. Section~\ref{sec:model_examples} is devoted to a presentation
of the main examples embedded in this setting. The main objects (such
as the empirical risk) and the basic deviation inequalities are
described in Section~\ref{sec:construction}. In Section~\ref{sec:erm},
we give risk bounds for the empirical risk minimization (ERM)
algorithm. To that end, we provide useful uniform deviation
inequalities using the generic chaining mechanism introduced in
\cite{talagrand06} (see Theorem~\ref{thm:generic_chaining} and
Corollary~\ref{cor:maximal_bracket}), and we give a general risk bound
for the ERM in Theorem~\ref{thm:ERM_bound} and its
Corollary~\ref{cor:erm_finite_dim}. In Section~\ref{sec:learning}, we
adapt a popular aggregation algorithm (aggregation with exponential
weights) to our setup, and give an oracle inequality (see
Theorem~\ref{thm:oracle}). In Section~\ref{sec:adaptive_estimation},
we use the results from Sections~\ref{sec:erm} and~\ref{sec:learning}
to construct estimators that adapt to the smoothness and to the
structure of the intensity. We compute the convergence rates of the
estimators, that are minimax optimal over anisotropic Besov
balls. Section~\ref{sec:main_proofs} contains the proofs. Some useful
results and tools are recalled in the Appendices.

\subsection{The model}

Let $(\Omega, \mathcal F, P)$ be a probability space and $(\mathcal
F_t)_{t \geq 0}$ a filtration satisfying the usual conditions, see
\cite{JacodShiryaev}. Let $N$ be a marked counting process with
compensator $\Lambda$ with respect to $(\mathcal F_t)_{t \geq 0}$, so
that $M = N - \Lambda$ is a $(\mathcal F_t)_{t \geq 0}$-martingale. We
assume that $N$ is a marked point process satisfying the \emph{Aalen
  multiplicative intensity model}. This means that $\Lambda$ writes
\begin{equation}
  \label{eq:aalen}
  \Lambda(t) = \int_0^t \alpha_0(u, X) Y(u) du
\end{equation}
for all $t \geq 0$, where:
\begin{itemize}
\item $\alpha_0$ is an unknown deterministic and nonnegative function
  called \emph{intensity};
\item $X \in \mathbb R^d$ is a $\mathcal F_0$-measurable random vector
  called \emph{covariates} or \emph{marks};
\item $Y$ is a predictable random process in $[0, 1]$.
\end{itemize}
With differential notations, this model can be written has
\begin{equation}
  \label{eq:model}
  d N(t) = \alpha_0(t, X) Y(t) dt + d M(t)
\end{equation}
for all $t \geq 0$ with the same notations as before, and taking $N(0)
= 0$. Now, assume that we observe $n$ i.i.d. copies
\begin{equation}
  \label{eq:whole_sample}
  D_n = \{ (X_i, N^i(t), Y^i(t)) : t \in [0, 1], 1 \leq i \leq n \}
\end{equation}
of $\{ (X, N(t), Y(t)) : t \in [0, 1] \}$. This means that we can
write
\begin{equation*}
  dN^i(t) = \alpha_0(t, X_i) Y^i(t) dt + d M^i(t)
\end{equation*}
for any $i = 1, \ldots, n$ where $M^i$ are independent $(\mathcal
F_t)_{t \geq 0}$-martingales. In this setting, the random variable
$N^i(t)$ is the number of observed failures during the time interval
$[0, t]$ of the individual~$i$.

The aim of the paper is to recover the intensity $\alpha_0$ on $[0,
1]$ based on the observation of the sample $D_n$. This general setting
includes several specific problems where the estimation of $\alpha_0$
is of importance for practical applications, see
Section~\ref{sec:model_examples}. In all what follows, we assume that
the support of $P_X$ is compact, but in order to simplify the
presentation, we shall assume the following.
\begin{assumption}
  \label{ass:support}
  The support of $P_X$ is $[0, 1]^d$\textup, and
  \begin{equation}
    \label{eq:sup_norm}
    \norm{\alpha}_\infty := \sup_{(t, x) \in [0, 1]^{d+1}}
    |\alpha(t, x)|
  \end{equation}
  is finite.
\end{assumption}
These assumptions on the model are very mild, excepted for the i.i.d
assumption of the sample, meaning that the individuals $i$ are
independent. Let us give several examples of interest that fit in this
general setting.

\subsection{Examples}
\label{sec:model_examples}

\subsubsection{Regression model for right-censored data}
\label{sec:censored_data}

Let $T$ be a nonnegative random variable (r.v.) and $X$ a vector of
covariates in $\mathbb R^d$. In this model, $T$ is not directly
observable: what we observe instead is
\begin{equation}
  \label{eq:censored}
  T^C := \min(T, C) \text{ and } \delta := I(T \leq C),
\end{equation}
where $C$ is a nonnegative random variable called
\emph{censoring}. This setting, where the data is \emph{right
  censored}, is of first importance in applications, especially in
medicine, biology and econometrics. In these cases, the r.v. $T$ can
represent the lifetime of an individual, the time from the the onset
of a disease to the healing, the duration of unemployment, etc. The
r.v. $C$ is often the time of last contact or the duration of
follow-up. In this model we assume the following mild assumption:
\begin{equation}
  \label{eq:weak_assumption}
  T \text{ and } C \text{ are independent conditionally to } X,
\end{equation}
which allows the censoring to depend on the covariates, see
\cite{MR2370107}. This assumption is weaker than the more common
  assumption that $T$ and $C$ are independent, see in particular
  \cite{MR1439707}.

In this case, the counting process writes
\begin{eqnarray*}
  N^i (t) = I(T_i^C \leq t, \delta_i=1) \text{ and }
  Y^i : Y^i (t) = I(T_i^C \geq t),
\end{eqnarray*}
see e.g.~\cite{ABGK}. In this setting, the intensity $\alpha_0$ is the
conditional hazard rate of $T$ given $X=x$, which is defined for all
$t > 0$ and $x \in \mathbb R^d$ by
\begin{equation*}
  \alpha_0(t, x) = \alpha_{T|X}(t, x) =
  \frac{f_{T|X}(t,x)}{1-F_{T|X}(t,x)},
\end{equation*}
where $f_{T|X}$ and $F_{T|X}$ are the conditional probability density
function (p.d.f.) and the conditional distribution function (d.f.) of
$T$ given $X$ respectively. The available data in this setting becomes
\begin{equation*}
  D_n := [ (X_i, T_i^C, \delta_i) : 1 \leq i \leq n],
\end{equation*}
where $(X_i, T_i^C, \delta_i)$ are i.i.d. copies of $(X, T^C,
\delta)$, where we assumed~\eqref{eq:weak_assumption}, namely $T_i$
and $C_i$ are independent conditionally to $X_i$ for $1 \leq i \leq
n$.

The nonparametric estimation of the hazard rate was initiated by
\cite{Beran}, \cite{STU}, \cite{DAB87}, \cite{McKeague} and \cite{LD}
extended his results.  Many authors have considered semiparametric
estimation of the hazard rate, beginning with \cite{Cox}, see
\cite{ABGK} for a review of the enormous literature on semiparemetric
models. We refer to \cite{Huang} and \cite{LNVdG} for some recent
developments. As far as we know, adaptive nonparametric estimation for
censored data in presence of covariates has only been considered in
\cite{brunel_comte_lacour07}, who constructed an optimal adaptive
estimator of the conditional density.

\subsubsection{Cox processes}
\label{sec:cox_processes}

Let $\eta^i, 1 \leq i \leq n$, be $n$ independent Cox processes on
$\mathbb R_+$, with mean-measure $A^i$ given by :
\begin{equation*}
  A^i(t) = \int_0^t \alpha(s, X_i) ds,
\end{equation*}
where $X_i$ is a vector of covariates in $\mathbb R^d$. This is a
particular case of longitudinal data, see e.g. Example~VII.2.15 in
\cite{ABGK}. The nonparametric estimation of the intensity of Poisson
processes without covariates has been considered in several papers. We
refer to \cite{Patricia1} for the adaptive estimation (using model
selection) for the intensity of nonhomogeneous Poisson processes in a
general space

\subsubsection{Regression model for transition intensities of Markov
  processes}
\label{sec:markov_processes}

Consider a $n$-sample of nonhomogeneous time-continuous Markov
processes $P^1,\dots,P^n$ with finite state space $\{1,\dots,k\}$ and
denote by $\lambda_{jl}$ the transition intensity from state $j$ to
state $l$. For an individual $i$ with covariate $X_i$, the
r.v. $N^i_{jl}(t)$ counts the number of observed direct transitions
from $j$ to $l$ before time $t$ (we allow the possibility of
right-censoring for example). Conditionally on the initial state, the
counting process $N^i_{jl}$ verifies the following Aalen
multiplicative intensity model:
\begin{eqnarray*}
  N^i_{jl}(t)=\int_0^t \lambda_{jl}(X_i,z)Y^i_j(z)dz+M^i(t) \text{ for
    all } t \geq 0,
\end{eqnarray*}
where $Y^i_j(t)=I(P^i(t-)=j)$ for all $t \geq 0$, see \cite{ABGK} or
\cite{Jacobsen}. This setting is discussed in \cite{ABGK}, see Example
VII.11 on mortality and nephropathy for insulin dependent diabetics.

We finally cite three papers, where the estimation of the intensity of
counting processes was considered, gathering as a consequence all the
previous examples, but in none of them the presence of covariates was
considered.  \cite{Ramlau} proposed a kernel-type estimator,
\cite{Gregoire} studied least squares cross-validation. More recently,
\cite{Patricia2} considered adaptive estimation by projection and
\cite{MR2449129} considered the adaptive estimation of the intensity
of a random measure by histogram-type estimators.

\subsection{Some notations}

From now on, we will denote by $L$ an absolute constant that can vary
from place to place (even in the same line), and by $c$ a constant
that depends on some parameters, that we shall indicate into
subscripts. In all of what follows, $D_n$ is an i.i.d. sample
satisfying model~\eqref{eq:model}, and we take $(X, (Y_t), (N_t))$
independent of $D_n$ that satisfies also model~\eqref{eq:model}. Note
that we will use both notations $(Z_t)_{t \geq 0}$ and $(Z(t))_{t \geq
  0}$ for a stochastic process $Z$. We denote by $\P^n[\cdot]$ the
joint law of $D_n$ and $\P[\cdot]$ the law of $(X, (Y_t), (N_t))$, and
by $\E^n[\cdot]$ and $\E[\cdot]$ the corresponding expectations.

\section{Main constructions and objects}
\label{sec:construction}

\subsection{An empirical risk}

Let $x \in \mathbb R^d$ and $(y_t)$, $(n_t)$ be functions $[0, 1]
\rightarrow \mathbb R^+$ with bounded variations, and let $\alpha :
[0, 1]^{d+1} \rightarrow \mathbb R^+$ be a bounded and predictable
function (that can eventually depend on $D_n$). We define the loss
function
\begin{equation*}
  \ell_\alpha(x, (y_t), (n_t)) = \int_0^1 \alpha(t, x)^2 y(t) dt - 2
  \int_0^1 \alpha(t, x) dn(t).
\end{equation*}
We define the least-squares type \emph{empirical risk} of $\alpha$ as:
\begin{align}
  \nonumber P_n(\ell_{\alpha}) &:= \frac1n \sum_{i=1}^n
  \ell_\alpha(X_i, (Y_t^i), (N_t^i)) \\
  \label{eq:empirial_risk}
  &= \frac{1}{n} \sum_{i=1}^n \int_0^1 \alpha(t, X_i)^2 Y^i(t) dt -
  \frac{2}{n} \sum_{i=1}^n \int_0^1 \alpha(t, X_i) d N^i(t).
\end{align}
This quantity measures the goodness-of-fit of $\alpha$ to the data
from $D_n$. It has been used in \cite{CGG} to perform model
selection. It is the empirical version of the \emph{theoretical risk}
\begin{align*}
  P(\ell_{\alpha}) &:= \E[ \ell_\alpha(X, (Y_t), (N_t)) \,|\, D_n] \\
  &=\E \Big[ \int_0^1 \alpha(t, X)^2 Y(t) dt - 2 \int_0^1 \alpha(t, X)
  d N(t) \,\big|\, D_n \Big].
\end{align*}
This risk is natural in this model. Indeed, if $\alpha$ is independent
of $D_n$, we have in view of~\eqref{eq:model}, since $M(t)$ is
centered:
\begin{align}
  \nonumber P(\ell_{\alpha}) &= \E \Big[ \int_0^1 ( \alpha(t, X)^2 - 2
  \alpha(t, X) \alpha_0(t, X) ) Y(t) dt \Big] - 2 \E\Big[ \int_0^1
  \alpha(t, X) d M(t) \Big] \\
  \label{eq:norm_risk}
  &= \norm{\alpha}^2 - 2 \prodsca{\alpha}{\alpha_0} \\
  &= \norm{\alpha -\alpha_0}^2 - \norm{\alpha_0}^2,
\end{align}
where:
\begin{align}
  \label{eq:norm}
  \prodsca{\alpha}{\alpha_0} &:= \int_{\mathbb R^d} \int_0^1 \alpha(t,
  x) \alpha_0(t, x) \E[Y(t) | X=x] dt P_X(dx), \quad \norm{\alpha}^2
  := \prodsca{\alpha}{\alpha}.
\end{align}
This is an inner product with respect to the bounded measure (it is
smaller than $1$)
\begin{equation}
  \label{eq:measure_mu}
  d \mu(t, x) := \E[Y(t) | X=x] dt P_X(dx).
\end{equation}
We will denote by $\L^2(\mu)$ the corresponding Hilbert space, and
define $\L^\infty(\mu)$ as the subset of $\L^2(\mu)$ consisting of
functions $\alpha$ such that $\norm{\alpha}_\infty < +\infty$.

In view of \eqref{eq:norm_risk}, $P(\ell_\alpha) - P(\ell_{\alpha_0})$
(called \emph{excess risk}) is equal to $\norm{\alpha -
  \alpha_0}^2$. As a consequence, $\alpha_0$ minimizes $\alpha \mapsto
P(\ell_\alpha)$, so a natural way to recover $\alpha_0$ is to take a
minimizer of $\alpha \mapsto P_n(\ell_\alpha)$. This is the basic idea
of empirical risk minimization, for which we propose risk bounds in
Section~\ref{sec:erm} below. Let us define the \emph{empirical norm}
\begin{equation}
  \label{eq:empirical_norm}
  \norm{\alpha}_n^2 := \frac{1}{n} \sum_{i=1}^n \int_0^1 \alpha(t,
  X_i)^2 Y^i(t) dt,
\end{equation}
so that we have $\E^n \norm{\alpha}_n^2 = \norm{\alpha}^2$ if $\alpha$
is deterministic. Note that $\norm{\alpha}_n \leq
\norm{\alpha}_\infty$ and $\norm{\alpha} \leq \norm{\alpha}_\infty$.
An important fact is that~\eqref{eq:model} entails
\begin{equation}
  \label{eq:empirical_norm_risk}
  P_n(\ell_\alpha) - P_n(\ell_{\alpha_0}) = \norm{\alpha - \alpha_0}_n^2 -
  \frac{2}{\sqrt{n}} Z_n(\alpha - \alpha_0),
\end{equation}
where $Z_n(\cdot)$ is given by
\begin{equation}
  \label{eq:empirical_process_training}
  Z_n(\alpha) = \frac{1}{\sqrt{n}} \sum_{i=1}^n \int_0^1 \alpha(t,
  X_i) d M^i(t),
\end{equation}
where $M^i$ are the independent copies of the martingale innovation
from~\eqref{eq:model}. The
decomposition~\eqref{eq:empirical_norm_risk} will be of importance in
the analysis of the problem.

\begin{remark}[Regression model for right-censored data]
  In the problem of censored survival times with covariates, see
  Section~\ref{sec:censored_data}, the semi-norm of estimation becomes
  \begin{equation*}
    \norm{\alpha}^2 = \int \int_0^1 \alpha(t, x)^2 \bar
    H_{T^C|X}(t, x) dt P_X(dx),
  \end{equation*}
  where $\bar H_{T^C|X}(t, x) := \P[ T^C > t | X = x]$, and where
  by~\eqref{eq:censored} and \eqref{eq:weak_assumption}:
  \begin{align*}
    \bar H_{T^C|X}(t, x) = \P[T > t | X = x] \P[C > t| X = x].
  \end{align*}
  This weighting of the norm is natural and, somehow, unavoidable in
  models with censored data. The same normalization can be found, for
  instance, in the Dvoretzky-Kiefer-Wolfowitz concentration inequality
  for the Kaplan-Meier estimator (without covariates), see Theorem~1
  in \cite{blm99}.
\end{remark}

\subsection{Deviation inequalities}

Let us denote by $\langle Z \rangle$ the predictable variation of a
random process $Z$. Note that we have, using
Assumption~\ref{ass:support}:
\begin{equation}
  \label{eq:variation_norm}
  \langle Z_n(\alpha) \rangle = \frac 1n \sum_{i=1}^n \int_0^1
  \alpha(t, X_i)^2 \alpha_0(t, X) Y^i(t) dt \leq \norm{\alpha_0}_\infty
  \norm{\alpha}_n^2.
\end{equation}
A useful result is then the following. First, introduce, for $\delta >
0$,
\begin{equation*}
  \psi_{n, \delta}(h) := \log \E^n [ e^{h Z_n(\alpha)} 
  \ind{\langle Z_n(\alpha) \rangle \leq \delta^2} ]
\end{equation*}
and the Cram\'er transform $\psi_{n, \delta}^*(z) := \sup_{h > 0} ( hz
- \psi_{n, \delta}(h))$.
\begin{proposition}
  \label{prop:deviation_martingale}
  For any bounded $\alpha$ and any $z, \delta > 0$, the following
  inequality holds\textup:
  \begin{equation}
    \label{eq:cramer_minoration}
    \psi_{n, \delta}^*(z) \geq \frac{n \delta^2
    }{\norm{\alpha}_\infty^2} g\Big( \frac{z
      \norm{\alpha}_\infty }{\delta^2 \sqrt{n}} \Big),
  \end{equation}
  where $g(x) := (1+x) \log(1+x) - x$.
\end{proposition}
This result and the deviation inequalities stated below are related to
standard results concerning martingales with jumps, see
\cite{lipstershiryayev}, \cite{vandegeer95} or \cite{Patricia2}, among
others. For the sake of completeness we give a proof of
Proposition~\ref{prop:deviation_martingale} in
Section~\ref{sec:main_proofs}. From the
minoration~\eqref{eq:cramer_minoration}, we can derive several
deviation inequalities. Using the Cram\'er-Chernoff bound
$\P^n[Z_n(\alpha) > z, \langle Z_n(\alpha) \rangle \leq \delta^2] \leq
\exp( -\psi_{n, \delta}^*(z) )$, we obtain the following Benett's
inequality:
\begin{equation*}
  \P^n[ Z_n(\alpha) > z, \langle Z_n(\alpha) \rangle \leq \delta^2
  ] \leq\exp\Big( -\frac{n
    \delta^2}{\norm{\alpha}_\infty^2}
  g\Big( \frac{z \norm{\alpha}_\infty }{
    \delta^2 \sqrt{n}} \Big) \Big)
\end{equation*}
for any $z > 0$. As a consequence, since $g(x) \geq 3x^2 / (2(x + 3))$
for any $x \geq 0$, we obtain the following Bernstein's inequality:
\begin{equation}
  \label{eq:bernstein1}
  \P^n[ Z_n(\alpha) > z, \langle Z_n(\alpha) \rangle \leq
  \delta^2 ] \leq \exp
  \Big( -\frac{z^2}{ 2 ( \delta^2 + z \norm{\alpha}_\infty / (3
    \sqrt{n}) } \Big).
\end{equation}
Another useful Bernstein's inequality can be derived using the
following trick from~\cite{birge_massart98}: since $g(x) \geq g_2(x)$
for any $x \geq 0$ where $g_2(x) = x+1 - \sqrt{1 + 2x}$, and since
$g_2^{-1}(y) = \sqrt{2y} + y$, we have
\begin{equation}
  \label{eq:bernstein2}
  \P^n\Big[ Z_n(\alpha) > \delta \sqrt{2 x} + \frac{
    \norm{\alpha}_\infty x }{ \sqrt{n}}, \,\, \langle Z_n(\alpha)
  \rangle \leq \delta^2 \Big] \leq \exp(-x)
\end{equation}
for any $x > 0$. Note that from~\eqref{eq:cramer_minoration}, we can
derive a uniform deviation inequality. Consider a family $(Z_n(\alpha)
: \alpha \in A)$, where $A$ is a set of bounded functions with finite
cardinality $N$. Since $\psi_{n, \delta}^{* -1}(z) \leq z
\norm{\alpha}_\infty / \sqrt{n} + \delta \sqrt{2 z}$ (see above) we
have, using Pisier's argument (see Section~2 in~\cite{massart03}),
that
\begin{align}
  \nonumber \P^n\Big[ Z_n(\alpha) > \delta \sqrt{2 (\ln N + x)} &+
  \frac{ \norm{\alpha}_\infty (\ln N + x) }{ \sqrt{n}}, \,\, \langle
  Z_n(\alpha) \rangle \leq \delta^2 \text{ for some } \alpha \in A
  \Big] \\ \label{eq:maximal_finite} &\leq \exp(-x).
\end{align}
In view of the next Lemma, we can remove the event $\{ \langle
Z_n(\alpha) \rangle \leq \delta^2 \}$ from the previous
inequalities. Indeed, a consequence of~\eqref{eq:bernstein2} is the
following.
\begin{lemma}
  \label{lem:tricky_deviation}
  If $\alpha$ is bounded, we have for any $x > 0$\textup:
  \begin{equation*}
    \P^n \Big[ Z_n(\alpha) \geq c \norm{\alpha} \sqrt{x} +
    (c + 1) \frac{\norm{ \alpha }_\infty x}{\sqrt n} \Big] \leq 2
    \exp(-x),
  \end{equation*}
  where $c = c_{\norm{\alpha_0}_\infty} := [\sqrt 2(\sqrt 2 + 1)
  \norm{\alpha_0}_\infty]^{1/2}$.
\end{lemma}

\begin{proof}
  Since $\E[(\int_0^1 \alpha(t, X)^2 Y(t) dt)^2] \leq
  \norm{\alpha^2}^2$ and $\norm{\int_0^1 \alpha(t, X)^2 Y(t)
    dt}_\infty \leq \norm{\alpha}_\infty^2$, Bernstein's inequality
  for the deviation of the sum of i.i.d. random variables gives:
  \begin{equation}
    \label{eq:bernstein_norm}
    \P^n\Big[ \norm{\alpha}_n^2 - \norm{\alpha}^2 >  
    \frac{\norm{\alpha^2} \sqrt{2x}}{\sqrt n} +  \frac
    {\norm{\alpha}_\infty^2 x} n \Big] \leq \exp(-x).
  \end{equation}
  Take $\delta_{n, x}^2 := \norm{\alpha}^2 + \norm{\alpha^2} \sqrt
  {2x} / \sqrt n + \norm{\alpha}_\infty^2 x / n$. We have $\P[
  \norm{\alpha}_n^2 > \delta_{n,x}^2 ] \leq \exp(-x)$ and
  \begin{equation*}
    \delta_{n,x} \sqrt {2 \norm{\alpha_0}_\infty x} + \frac{
      \norm{\alpha}_\infty x}{\sqrt n} \leq c_{\norm{\alpha_0}_\infty}
    \sqrt{x} \norm{\alpha} + (c_{\norm{\alpha_0}_\infty} + 1) \frac{
      \norm{\alpha}_\infty x}{\sqrt n}.
  \end{equation*}
  Now, use~\eqref{eq:variation_norm} and \eqref{eq:bernstein2} to
  obtain
  \begin{equation*}
    \P \Big[ Z_n(\alpha) \geq \delta_{n,x} \sqrt {2
      \norm{\alpha_0}_\infty x} + \frac{ \norm{\alpha}_\infty x}{\sqrt
      n}, \norm{\alpha}_n^2 \leq \delta_{n, x}^2 \Big] \leq \exp(-x)
  \end{equation*}
  for any $x > 0$. This concludes the proof of the Lemma, by a
  decomposition over $\{ \norm{\alpha}_n > \delta_{n,x} \} $ and $\{
  \norm{\alpha}_n \leq \delta_{n,x} \}$.
\end{proof}

These deviation inequalities are the starting point of the proof of
risk bounds for the algorithm of empirical risk minimization
(ERM). Such a bound is given in Section~\ref{sec:erm} below, see
Theorem~\ref{thm:ERM_bound}. It requires a generalization of the
bound~\eqref{eq:maximal_finite} to a general set $A$, which is given
in Section~\ref{sec:chaining}.

\section{Empirical risk minimization}
\label{sec:erm}

The very basic idea of empirical risk minimization (ERM) is the
following. Since $\alpha_0$ minimizes the risk $\alpha \mapsto
P(\ell_\alpha)$, a natural estimate of $\alpha_0$ is a minimizer of
the empirical risk $\alpha \mapsto P_n(\ell_\alpha)$ over some set of
function $A$, usually called a \emph{sieve}. There is hope that such
an empirical minimizer is close to $\alpha_0$, at least if $\alpha_0$
is not far from $A$ and if $(P - P_n)(\ell_\alpha)$ is small (more
details below). Also known as M-estimation, this algorithm has been
studied extensively, see for instance \cite{birge_massart98},
\cite{vapnik00}, \cite{van_de_geer00}, \cite{massart03},
\cite{bartlett_mendelson06}, among many others.

If no minimizer of the empirical risk exists, we can simply consider,
as this is usually done in the literature, a $\rho$-minimizer
according to the following definition.
\begin{definition}[$\rho$-ERM]
  \label{def:erm}
  Let $\rho > 0$ be fixed. A $\rho$-Empirical Risk Minimizer
  ($\rho$-ERM) is an estimator $\bar \alpha_n \in A$ satisfying
  \begin{equation*}
    P_n(\ell_{\bar \alpha_n}) \leq \rho + \inf_{\alpha \in A}
    P_n(\ell_{\alpha}),
  \end{equation*}
  where $P_n(\ell_{\alpha})$ is the empirical
  risk~\eqref{eq:empirial_risk}.
\end{definition}
For what follows, one can take $\rho = 1/n$, since typically, the risk
of $\bar \alpha_n$ is larger than that. To prove a risk bound for the
ERM, one usually needs a deviation inequality for
\begin{equation*}
  \zeta_n(A) := \sup_{\alpha \in A} (P - P_n)(\ell_\alpha).
\end{equation*}
However, when $A$ is not countable, $\zeta_n$ may be not
measurable. This is not a problem since we can always consider the
outer expectation in the statement of the deviation (see
\cite{van_der_vaart_wellner96}), or simply assume the following.
\begin{assumption}
  \label{ass:countable}
  There is a countable subset $A'$ of $A$ such that almost surely,
  \begin{equation*}
    \sup_{\alpha \in A'} P_n(\ell_\alpha) = \sup_{\alpha \in A}
    P_n(\ell_\alpha).
  \end{equation*}
  Moreover\textup, assume that there is $b > 0$ such that
  $\norm{\alpha}_\infty \leq b$ for every $\alpha \in A$.
\end{assumption}
The map $\alpha \mapsto P_n(\ell_\alpha)$ is continuous over $\mathsf
C([0, 1]^{d+1})$ endowed with the norm $\norm{\cdot}_\infty$. So,
given that $A \subset \mathsf C([0, 1]^{d+1})$, the first part of
Assumption~\ref{ass:countable} is met. Note that this embedding holds
in the examples considered in Section~\ref{sec:adaptive_estimation}.
The second part is rather unpleasant, but mandatory if no extra
assumption is made on $A$, and since an $\L^2$ metric is considered
for the estimation of $\alpha_0$.

From now on, we take $\alpha_*$ such that $P(\ell_{\alpha_*}) =
\inf_{\alpha \in A} P(\ell_\alpha)$ (if no such $\alpha_*$ exists, we
can simply consider $\alpha_*$ such that $P(\ell_{\alpha_*}) \leq
\inf_{\alpha \in A} P(\ell_\alpha) + \rho$). Note that $\alpha_*$ may
not be unique at this point, we just pick one of the minimizers. The
function $\alpha_*$ is usually called the \emph{target} function, or
the \emph{oracle} in learning theory, see \cite{cucker_smale02} for
instance.

\subsection{Peeling}

A common way to prove a risk bound for the ERM uses the idea of
\emph{localization} or \emph{peeling} (see for instance
\cite{massart03}, Lemma~4.23 and \cite{van_de_geer00, van_de_geer07},
among others). The idea presented here is very close to these
references. First, do a shift: take $\epsilon > 0$, and use the fact
that $\bar \alpha_n$ is a $\rho$-ERM to obtain
\begin{align*}
  P(\ell_{\bar \alpha_n}) - P(\ell_{\alpha_*}) &\leq (1 + \epsilon)
  \rho + P(\ell_{\bar \alpha_n}) - P(\ell_{\alpha_*}) - (1 + \epsilon)
  (P_n(\ell_{\bar \alpha_n}) - P_n(\ell_{\alpha_*})) \\
  &\leq (1 + \epsilon) \rho + \xi_{n, \epsilon}(A),
\end{align*}
where
\begin{equation*}
  \xi_{n, \epsilon}(A) := \sup_{\alpha \in A} \Big( (1 + \epsilon)
  (P - P_n) (\ell_\alpha - \ell_{\alpha_*}) - \epsilon P(\ell_{\alpha}
  - \ell_{\alpha_*}) \Big).
\end{equation*}
Then, for some constants $\delta > 0$ and $q > 1$, decompose the
supremum over $A$ into suprema over annuli $A_j(\delta)$, where
$A(\delta) = \{ \alpha \in A : P(\ell_\alpha) - P(\ell_{\alpha_*})
\leq \delta\}$, and for $j \geq 1$, $A_j(\delta) = \{ \alpha \in A :
q^j \delta < P(\ell_\alpha) - P(\ell_{\alpha_*}) \leq q^{j+1}
\delta\}$. Assume for the moment that there exists an increasing
function $\psi : \mathbb R^+ \rightarrow \mathbb R^+$ and
$\delta_{\min}> 0$ such that for any $x > 0$ and $\delta >
\delta_{\min}$, we have with a probability larger than $1 - L e^{-x}$:
\begin{equation}
  \label{eq:concentration}
  \sup_{\alpha \in A(\delta)} (P - P_n)(\ell_\alpha - \ell_{\alpha_*})
  \leq \frac{\psi(\delta)(1 + \sqrt x \vee x)}{\sqrt n}.
\end{equation}
Such an inequality will be proved in Section~\ref{sec:chaining} below.
It entails that, with a probability larger than $1 - L e^{-x}$:
\begin{equation}
  \label{eq:peeling_argument}
  \xi_{n, \epsilon}(A) \leq (1 + \epsilon) \frac{\psi(\delta)(1 +
    \sqrt x \vee x)}{\sqrt n} +
  \sup_{j \geq 1} \Big(  (1+\epsilon) \frac{\psi(q^{j+1} \delta)(1 + \sqrt x
    \vee x)}{\sqrt n} - \epsilon q^j \delta \Big).
\end{equation}
Assume further that $\psi$ is continuous, increasing, such that
$\delta \mapsto \psi(\delta) / \delta$ is decreasing and $\psi^{-1}$
is strictly convex. We can define the convex conjugate of $\psi^{-1}$
as
\begin{equation}
  \label{eq:convex_conjuguate}
  \psi^{-1*}(\delta) := \sup_{x > 0} \{ x \delta - \psi^{-1}(x) \}.
\end{equation}
The following Lemma comes in handy to choose a parameter $\delta$ that
kills the second term in the right hand side of
\eqref{eq:peeling_argument}.
\begin{lemma}
  \label{lem:convex_trick}
  Let $\psi : \mathbb R^+ \rightarrow \mathbb R^+$ be a continuous and
  increasing function and assume that $\psi^{-1}$ is strictly convex.
  If $\delta := \psi^{-1 *}(2 x / y)$\textup, we have
  \begin{equation*}
    x \psi(\delta) \leq y \delta
  \end{equation*}
  for any $x, y > 0$.
\end{lemma}
\begin{proof}
  Simply write
  \begin{align*}      
    x \psi(\delta) = \frac{y}{2} \frac{2x}{y} \psi\Big( \psi^{-1*}
    \Big( \frac{2x}{y} \Big) \Big) \leq \frac{y}{2} \Big( \psi^{-1*}
    \Big(\frac{2x}{y} \Big) + \psi^{-1*} \Big( \frac{2x}{y} \Big)
    \Big),
  \end{align*}
  where the trick is to use the fact that $u v \leq \psi^{-1*}(u) +
  \psi^{-1}(v)$ for any $u, v > 0$.
\end{proof}
Using Lemma~\ref{lem:convex_trick} and the fact that $\psi(q^{j+1}
\delta) \leq q^{j+1} \psi(\delta)$, we obtain that for the choice
\begin{equation*}
  \delta_{n, \epsilon}(x) := \psi^{-1*} \Big( \frac{2 q (1+\epsilon) (1
    + \sqrt x \vee x)}{\epsilon \sqrt n} \Big),
\end{equation*}
we have, with a probability larger than $1 - L e^{-x}$:
\begin{equation*}
  P(\ell_{\bar \alpha_n}) - P(\ell_{\alpha_*}) \leq (1 + \epsilon)
  \rho + \epsilon \delta_{n, \epsilon}(x).
\end{equation*}
We have proved the following result.
\begin{proposition}[Peeling]
  \label{prop:peeling}
  Assume that~\eqref{eq:concentration} holds for any $\delta >
  \delta_{\min}$\textup, where $\psi : \mathbb R^+ \rightarrow \mathbb
  R^+$ is a continuous and increasing function such that $\psi^{-1}$
  is strictly convex and $\delta \mapsto \psi(\delta) / \delta$ is
  decreasing. If $\bar \alpha_n$ is a $\rho$-ERM according to
  Definition~\ref{def:erm}\textup, we have for any $x > 0$\textup:
  \begin{equation*}
    P(\ell_{\bar \alpha_n}) \leq P(\ell_{\alpha_*}) + (1 + \epsilon)
    \rho + \epsilon \delta_{n, \epsilon}(x)
  \end{equation*}
  with probability larger than $1 - L e^{-x}$\textup, where
  \begin{equation*}
    \delta_{n, \epsilon}(x) := \psi^{-1*} \Big( \frac{2(1+\epsilon) q (1
      + \sqrt x \vee x)}{\epsilon \sqrt n} \Big) \vee \delta_{\min}.
  \end{equation*}  
\end{proposition}
In the next section, we prove Inequality~\eqref{eq:concentration}
using the generic chaining mechanism, under an assumption on the
complexity of $A$.

\subsection{Generic chaining}
\label{sec:chaining}

The \emph{generic chaining} technique, which is introduced in
\cite{talagrand06} is, in our setting, a nice way to prove
\eqref{eq:concentration}. It is based on the $\gamma_\nu(A, d)$
functional (see below) which is an alternative to Dudley's entropy
integral (see \cite{dudley78} for instance). The idea is to decompose
$A$ using an approximating sequence of partitions, instead of nets
with decreasing radius, as this is done in the standard chaining
method. Let us briefly recall some necessary notions that can be found
in details in \cite{talagrand06}, Chapter~1.

Let $(A, d)$ be a metric space ($d$ can be a semi-distance). Denote by
$\Delta(A, d) := \sup_{a, b \in B} d(a, b)$ the diameter of $A$. An
\emph{admissible} sequence of $A$ is an increasing sequence $(\mathcal
A_j)_{j \geq 0}$ of partitions of $A$ (every set of $\mathcal A_{j+1}$
is included in a set of $\mathcal A_j$) such that $| \mathcal A_j |
\leq 2^{2^j}$ and $|\mathcal A_0| = 1$. If $a \in A$, we denote by
$A_j(a)$ the unique element of $\mathcal A_j$ that contains $a$. For
$\nu > 0$, define the function
\begin{equation}
  \label{eq:gamma_function}
  \gamma_\nu(A, d) := \inf \sup_{a \in A} \sum_{j \geq 0} 2^{j / \nu}
  \Delta( A_j(a), d )
\end{equation}
where the infimum is taken among all admissible sequence of $A$. This
quantity is an alternative (and an improvement, see
\cite{talagrand06}, in particular Theorem~3.3.2) of the Dudley's
entropy integral (see for instance
\cite{van_der_vaart_wellner96}). Indeed, we have:
\begin{equation}
  \label{eq:gamma_dudley}
  \gamma_\nu(A, d) \leq L \int_0^{\Delta(A, d)} ( \log N(A,
  \varepsilon, d) )^{1 / \nu} \rm d \varepsilon,
\end{equation}
where % $L_\nu := [ ( 1 - 2^{-1 / \nu}) \log 2]^{-1}$ and
$N(A, \varepsilon, d)$ is the \emph{covering number} of $A$, namely
the smallest integer $N$ such that there is $B \subset A$ satisfying
$|B| \leq N$ and $d(a, B) \leq \varepsilon$ for any $a \in
A$. 

Introduce $d_2(a, b) := \norm{a - b}$ where $\norm{\cdot}$ is the
semi-norm given by~\eqref{eq:norm} and $d_\infty(a, b) = \norm{a -
  b}_\infty$, where $\norm{\cdot}_\infty$ is the uniform norm
\eqref{eq:sup_norm}. Using the generic chaining argument, we obtain
the following deviation inequality.
\begin{theorem}
  \label{thm:generic_chaining}
  Grant Assumptions~\ref{ass:support} and~\ref{ass:countable}. For any
  $x > 0$\textup, we have
  \begin{equation}
    \label{eq:general_maximal}
    \sup_{\alpha \in A} \sqrt n (P - P_n)(\ell_\alpha - \ell_{\alpha_*})
    \leq c \Big( \gamma_2(A, d_2) (1 + \sqrt x) + \gamma_1(A,
    d_\infty) \frac{1 + x}{ \sqrt{n} } \Big)
  \end{equation}
  with a probability larger than $1 - L e^{-x}$\textup, where $c =
  c_{b, \norm{\alpha_0}_\infty} = 4(b + \norm{\alpha_0}_\infty) + 2
  ([\sqrt 2(\sqrt 2 + 1) \norm{\alpha_0}_\infty]^{1/2} + 1)$
  \textup(and $L \approx 1.545433$\textup).
\end{theorem}
The proof of Theorem~\ref{thm:generic_chaining} is given in
Section~\ref{sec:main_proofs} below. In~\eqref{eq:general_maximal},
the function $\gamma_2$ is related to the subgaussian term of the
Bernstein inequality~\eqref{eq:bernstein2}, while $\gamma_1$ is
related to the subexponential term. However, if we have an extra
condition on the complexity of $A$, it is possible to ``remove'' the
$\gamma_1$ term from~\eqref{eq:general_maximal}. This is called the
\emph{adaptive truncation argument}, which is related to the use of
brackets (instead of balls) to construct a covering of $A$.

\subsection{Brackets}

Entropy with bracketing has been introduced by \cite{dudley78}. The
adaptive truncation argument was introduced by \cite{bass85} for
partial sum process and \cite{ossiander87} for the empirical
process. We refer to \cite{van_de_geer00} (in particular the proof of
Theorem~8.13) herein and \cite{massart03} (see the proof of
Theorem~6.8) for the use of this technique with statistical
applications in mind. In the context of generic chaining, bracketing
can defined as follows. Following \cite{talagrand06} (see in
particular Theorem~2.7.10), we consider
\begin{equation}
  \label{eq:gamma_bracket}
  \gamma^{[]}(A) := \inf \sup_{a \in A} \sum_{j \geq 0} 2^{j/2}
  \norm{\mathscr B_{A_j(a)}},
\end{equation}
where the infimum is taken among all admissible sequences of $A$,
where we recall that $\norm{\cdot}$ is defined by~\eqref{eq:norm}, and
where
\begin{equation*}
  \mathscr B_A(z) := \sup_{a, a' \in A} | a(z) - a'(z)|
\end{equation*}
for any $z \in [0, 1]^{d+1}$. If $a^L, a^U \in A$, the \emph{bracket}
$[a^L, a^U]$ is the band
\begin{equation*}
  [a^L, a^U] := \{ a \in A : a^L \leq a \leq a^U \text{ pointwise} \}.
\end{equation*}
The quantity $\norm{a^U - a^L}$ is the \emph{diameter} of the
bracket. We denote by $N^{[]}(A, \epsilon)$ the minimal number of
brackets with diameter not larger than $\epsilon$ necessary to cover
$A$. Analogously to~\eqref{eq:gamma_dudley}, one has
\begin{equation}
  \label{eq:gamma_bracket_dudley}
  \gamma^{[]}(A) \leq L \int_0^{\Delta(A, d_\infty)} (\log N^{[]}(A,
  \epsilon))^{1/2} \mathrm d \epsilon.
\end{equation}
Entropy with bracketing is a refinement of $\L^\infty$-entropy, that
can be suitable for some class of functions, for instance functions
with uniformly bounded variation, see for instance \cite{vdg93} and
\cite{blm99}. In our setting, it is useful to ``remove'' the
$\gamma_1$ term from~\eqref{eq:general_maximal}, thanks to the
following result, which is Talagrand's version of the adaptive
truncation argument.
\begin{theorem}[\cite{talagrand06}, Theorem~2.7.11]
  \label{thm:tala_bracket}
  Let $A$ be a countable set of measurable functions, and let $u >
  0$. If $\gamma^{[]}(A) \leq \Gamma$\textup, we can find two sets
  $A_1, A_2$ with the following properties\textup:
  \begin{itemize}
  \item $\gamma_2(A_1, d_2) \leq L \Gamma, \quad \gamma_1(A_1,
    d_\infty) \leq L u \Gamma,$
  \item $\gamma_2(A_2, d_2) \leq L \Gamma, \quad \gamma_1(A_2,
    d_\infty) \leq L u \Gamma,$
  \item for any $a \in A_2,$ we have $a \geq 0$ and $\norm{a}_1 \leq L
    \Gamma / u,$ and
    \begin{equation*}
      A \subset A_1 + A_2', \text{ where } A_2' = \{ a' : \exists a \in
      A_2, |a'| \leq a \}.
    \end{equation*}  
  \end{itemize}
\end{theorem}
Indeed, an immediate consequence of
Proposition~\ref{thm:generic_chaining} and
Theorem~\ref{thm:tala_bracket} (simply take $u = \sqrt n$ in
Theorem~\ref{thm:tala_bracket}) is the following.
\begin{corollary}
  \label{cor:maximal_bracket}
  Grant Assumptions~\ref{ass:support} and~\ref{ass:countable}. For any
  $x > 0$\textup, we have
  \begin{equation*}
    \sup_{\alpha \in A} \sqrt n (P - P_n)(\ell_\alpha -
    \ell_{\alpha_*}) \leq c \gamma^{[]}(A) (1 + \sqrt x \vee x)
  \end{equation*}
  with a probability larger than $1 - L e^{-x}$\textup, where $c =
  c_{b, \norm{\alpha_0}_\infty}$ is the same as in
  Theorem~\ref{thm:generic_chaining}.
\end{corollary}

\subsection{A risk bound for the ERM}
\label{sec:erm_risk_bound}

Corollary \ref{cor:maximal_bracket} is close to the concentration
inequality~\eqref{eq:concentration} required in the peeling argument,
see Proposition~\ref{prop:peeling} above. However, note that the
peeling was done using sets $A(\delta) = \{ \alpha \in A :
P(\ell_\alpha) - P(\ell_{\alpha_*}) \leq \delta\}$ for $\delta > 0$,
while we can bound from above the entropy (and consequently the
functionals $\gamma$ and $\gamma^{[]}$) of balls $B(\delta) = \{
\alpha \in A : \norm{\alpha - \alpha_*} \leq \delta\}$ using a
standard result (see Section~\ref{sec:A_finite_dimension}). Hence, it
will be convenient to work under the following assumption.
\begin{assumption}
  \label{ass:margin}
  Assume that $\norm{\alpha - \alpha_*}^2 \leq P(\ell_\alpha) -
  P(\ell_{\alpha_*})$ for every $\alpha \in A$.
\end{assumption}
This assumption is a bit stronger than the standard \emph{margin
  assumption}, see \cite{mammen_tsybakov99, tsybakov04}, or the
\emph{$\beta$-Bernstein condition}, see \cite{bartlett_mendelson06}
[Note that here $\beta = 1$, as in most statistical models, see
\cite{these_bobby}.] Indeed, let us prove that
Assumption~\ref{ass:margin} entails, together with
Assumptions~\ref{ass:support} and~\ref{ass:countable}:
\begin{equation}
  \label{eq:bernstein_condition_BM}
  P( (\ell_\alpha - \ell_{\alpha_*})^2 )  \leq c P(\ell_\alpha -
  \ell_{\alpha_*}) \; \text{ for every } \; \alpha \in A,
\end{equation}
where $c = c_{b, \norm{\alpha_0}_\infty} := 8 ((b +
\norm{\alpha_0}_\infty)^2 + \norm{\alpha_0}_\infty)$, which is the
$(1, c)-$Bernstein condition from~\cite{bartlett_mendelson06}. We have
using~\eqref{eq:model}:
\begin{align*}
  \ell_{\alpha}(X, (Y_t), (N_t)) = \ell_\alpha'(X, (Y_t)) - 2 \int_0^1
  \alpha(t, X) d M(t),
\end{align*}
where $\ell_\alpha'$ is the loss function
\begin{align*}
  \ell_\alpha'(x, (y_t)) := \int_0^1 \alpha(t, x)^2 y(t) dt - 2
  \int_0^1 \alpha(t, x) \alpha_0(t, x) y(t) dt,
\end{align*}
so the following decomposition holds:
\begin{align*}
  \ell_\alpha(X, (Y_t), &(N_t)) - \ell_{\alpha_*}(X, (Y_t), (N_t)) \\
  &= \ell_\alpha'(X, (Y_t)) - \ell_{\alpha_*}'(X, (Y_t)) + 2 \int_0^1
  (\alpha_*(t, X) - \alpha(t, X)) dM(t) \\
  &= \int_0^1 ( \alpha(t, X) - \alpha_*(t, X)) (\alpha(t, X) +
  \alpha_*(t, X) - 2 \alpha_0(t, X) ) Y(t) dt \\
  &+ 2 \int_0^1 (\alpha_*(t, X) - \alpha(t, X)) dM(t).
\end{align*}
Hence, using Assumptions~\ref{ass:support} and~\ref{ass:countable}, we
have:
\begin{align*}
  P( (\ell_\alpha - \ell_{\alpha_*})^2 ) &\leq 8 (b +
  \norm{\alpha_0}_\infty)^2 \norm{\alpha - \alpha_*}^2 \\
  &+ 8 \E\Big[ \int_0^1 (\alpha_*(t, X) -
  \alpha(t, X))^2 \alpha_0(t, X) Y(t) dt \Big] \\
  &\leq 8 ((b + \norm{\alpha_0}_\infty)^2 + \norm{\alpha_0}_\infty)
  \norm{\alpha - \alpha_*}^2,
\end{align*}
and \eqref{eq:bernstein_condition_BM} follows using
Assumption~\ref{ass:margin}. Now, let us show that
Assumption~\ref{ass:margin} is mild: it is met when $A$ is convex, for
instance. The fact that convexity entails the margin assumption is
true in most statistical models, such a in regression, see for
instance \cite{lee_bartlett_williamson98}.
  \begin{lemma}
    \label{lem:margin}
    Grant Assumption~\ref{ass:support} and let $A$ be a convex class
    of functions bounded by $b > 0$. Then\textup,
    Assumption~\ref{ass:margin} is met.
  \end{lemma}
  \begin{proof}
    Since $A$ is convex and $P(\alpha_*) = \inf_{\alpha \in A}
    P(\ell_\alpha)$, we have $\prodsca{\alpha_* - \alpha_0}{\alpha_*,
      \alpha} \leq 0$ for any $\alpha \in A$, where we recall that the
    inner product is given by~\eqref{eq:norm}. This entails
    \begin{align*}
      P( \ell_\alpha - \ell_{\alpha_*} ) &= \norm{\alpha}^2 - 2
      \prodsca{\alpha}{\alpha_0} - \norm{\alpha_*}^2 + 2
      \prodsca{\alpha_*}{\alpha_0} \\
      &= \norm{\alpha - \alpha_*}^2 - 2 \prodsca{\alpha_* -
        \alpha_0}{\alpha_* - \alpha} \\
      &\geq \norm{\alpha - \alpha_*}^2. \qedhere
  \end{align*}
\end{proof}
We are now in position to state the following risk bound for the ERM,
under a condition on the complexity of $A$.
\begin{theorem}
  \label{thm:ERM_bound}
  % \label{ass:complexity}
  Grant Assumptions~\ref{ass:support},~\ref{ass:countable}
  and~\ref{ass:margin}. Assume that there is $\delta_{\min} > 0$ and a
  continuous and increasing function $\varphi : \mathbb R^+
  \rightarrow \mathbb R^+$ such that for any $\delta >
  \delta_{\min}$\textup, any $\alpha' \in A$ and any $\sqrt
  \delta$-ball $B(\sqrt \delta) = \{ \alpha \in A : \norm{\alpha -
    \alpha'}^2 \leq \delta \}$\textup, we have either\textup:
  \begin{equation*}
    \varphi(\delta) \geq \gamma_2(B(\sqrt \delta), d_2) + \frac{1}{
      \sqrt{n} } \gamma_1(B(\sqrt \delta), d_\infty) \; \text{ for any
    } \; \delta > \delta_{\min},
  \end{equation*}
    or\textup:
    \begin{equation*}
      \varphi(\delta) \geq \gamma^{[]}(B(\sqrt \delta)) \; \text{ for any
      } \; \delta > \delta_{\min}.
    \end{equation*}
    Assume further that $\varphi^{-1}$ is strictly convex and that
    $\delta \mapsto \varphi(\delta) / \delta$ is decreasing. Then, if
    $\bar \alpha_n$ is a $\rho$-ERM according to
    Definition~\ref{def:erm}\textup, we have for any $\epsilon > 0,$
    $x > 0$\textup:
    \begin{equation*}
      P(\ell_{\bar \alpha_n}) \leq P(\ell_{\alpha_*}) + (1 + \epsilon)
      \rho + \epsilon \delta_{n, \epsilon}(x)
    \end{equation*}
    with a probability larger than $1 - L e^{-x}$\textup, where
    \begin{equation*}
      \delta_{n, \epsilon}(x) := \varphi^{-1*} \Big( \frac{c
        (1+\epsilon) (1 + \sqrt x \vee x)}{\epsilon \sqrt n} \Big)
      \vee \delta_{\min},
    \end{equation*}  
    and $c = c_{b, \norm{\alpha_0}_\infty}$ is the same as in
    Theorem~\ref{thm:generic_chaining}.
  \end{theorem}
  \begin{proof}
    Because of Assumption~\ref{ass:margin}, we have $A(\delta) \subset
    B(\sqrt{\delta})$, so Inequality~\eqref{eq:concentration} is
    satisfied under the assumptions of the theorem with $\psi(\delta)
    = c \varphi(\delta)$, using Theorem~\ref{thm:generic_chaining} or
    Corollary~\ref{cor:maximal_bracket}. Hence, we can apply
    Proposition~\ref{prop:peeling}, which entails the Theorem since
    $\psi^{-1*}(x) = \varphi^{-1*}(c x)$.
  \end{proof}

  \begin{remark}[Comparison]
    This bound for the ERM is of the same nature as previous bounds
    for the ERM in more ``standard'' models, such as density,
    regression or classification. The rate given in
    Theorem~\ref{thm:ERM_bound} gives, on examples, the same rate (up
    to constants) as the one given in \cite{massart03} (see
    Theorem~8.3), for instance. Consider the situation where
    $\varphi(\delta) = c \delta^{\alpha}$ for $c > 0$ and $\alpha \in
    (0, 1)$ ($\varphi(\delta)$ is of order $\sqrt{D \delta}$ when $A$
    has a finite dimension, see Section~\ref{sec:A_finite_dimension}
    below). In this case, we have $\varphi^{-1 *}(x ) = (1 - \alpha)
    \alpha^{\alpha / (1 - \alpha)} (c x)^{1 / (1 - \alpha)}$, so
    $\delta_{n, \epsilon}(x)$ is of order $(c / \sqrt{n})^{1 / (1 -
      \alpha)}$. The rate $\varepsilon_*^2$ in the bound by Massart is
    solution to the equation $\sqrt n \varepsilon^2 =
    \varphi(\varepsilon^2)$, hence $\varepsilon_*^2 = (c / \sqrt n)^{1
      / (1 - \alpha)}$, and both rates have the same order.
  \end{remark}
  
  \begin{remark}[Talagrand's inequality]
    Usually, the complexity of the sieve $A$ is measured by the
    functional $\phi(B) = \sqrt n\, \E^n[ \sup_{\alpha \in B} (P -
    P_n)(\ell_{\alpha} - \ell_{\alpha_*})]$ where $B$ are balls in
    $A$, like in \cite{massart03} or spheres in $A$, see
    \cite{bartlett_mendelson06}. The rate is then the solution of a
    fixed point problem involving these functional, such as, roughly,
    the equation $\phi(B(\varepsilon)) = \sqrt n \varepsilon^2$ from
    \cite{massart03}. Note that the main tool in the proof of these
    results is Talagrand's deviation inequality, see \cite{massart00},
    \cite{rio01} or \cite{bousquet_phd}. In
    Theorem~\ref{thm:ERM_bound}, we were not able to state the bound
    with a rate defined in such a way. Indeed, we needed a
    ``stronger'' control on the complexity, given by the $\gamma$
    functionals, to define $\delta_{n, \epsilon}$. This is related to
    the fact that we cannot use a Talagrand's type deviation
    inequality in the general model~\eqref{eq:model} for
    \begin{equation*}
      \sup_{\alpha \in A} (P - P_n)(\ell_{\alpha} - \ell_{\alpha_*}) -
      \E^n[ \sup_{\alpha \in A} (P - P_n)(\ell_{\alpha} -
      \ell_{\alpha_*}) ].
    \end{equation*}
    Indeed, $\int_0^1 \alpha(t, X) d M(t)$ is not, in general, bounded
    (think of the Poisson process for instance, which is a particular
    case of Section~\ref{sec:cox_processes}).

    However, the story is different when $N(t)$ is bounded, such as in
    the models of regression for right-censored data, and of
    transition intensities of Markov processes (see
    Sections~\ref{sec:censored_data} and
    \ref{sec:markov_processes}). % In these examples, the counting
    % process $N(t)$ is bounded (by $1$ in the first case, by ?
    % \texttt{demander a Agathe...} in the other).
    It is then possible to use the strength of Talagrand's inequality,
    following the arguments from \cite{bartlett_mendelson06} (up to
    significant modifications, since the analysis is conducted in the
    regression model).
  \end{remark}
  
  A case of importance (particularly in practice) is when $A$ is
  included in a linear space $\bar A$ with a finite dimension $D$ (see
  \cite{birge_massart98} and \cite{massart03} for instance). Using the
  version of \cite{massart03} of a classical result concerning
  $\L^\infty$-coverings of a ball in such a space (see below), we can
  show that $\delta_{n, \epsilon}(x)$ is smaller than a quantity of
  order $D / n$. This will be useful to compute rates of convergence
  in Section~\ref{sec:adaptive_estimation} below.

  \subsection{When $A$ is finite dimensional}
  \label{sec:A_finite_dimension}

  Let us now consider the case where $A$ is a subset of some linear
  space $\bar A \subset \L^2 \cap \L^\infty[0, 1]^{d+1}$ with finite
  dimension $D$. Following~\cite{birge_massart98} and \cite{bbm99}, we
  can consider the \emph{$\L^\infty$-index}
\begin{equation*}
  r(\bar A) := \frac{1}{\sqrt D} \inf_{\psi} \sup_{\beta \neq 0}
  \frac{\norm{\sum_{\lambda \in \Lambda} \beta_\lambda
      \psi_\lambda}_\infty}{|\beta|_\infty},
\end{equation*}
where $|\beta|_\infty = \max_{\lambda \in \Lambda} | \beta_\lambda |$
and where the infimum is taken over all orthonormal basis $\{
\psi_\lambda : \lambda \in \Lambda \}$ of $\bar A$.

This index can be estimated for all the linear spaces usually chosen
as approximation spaces for adaptive estimation, see
\cite{birge_massart98} and \cite{bbm99}. In particular, if $\bar A$ is
spanned by a localized basis, then $r(\bar A)$ can be bounded
independently of $D$ (think of a wavelet basis for instance, more on
that in Section~\ref{sec:adaptive_estimation} below).

Using this index, we can derive a bound for $\gamma^{[]}(B(\sqrt
\delta))$. For any $\varepsilon \in (0, \delta]$, the following holds
(see \cite{massart03}, Lemma~7.14):
\begin{equation}
  \label{eq:entropy_finitedim}
  N( B(\delta), \varepsilon, d_\infty) \leq \Big(
  \frac{L r(\bar A) \delta}{\varepsilon} \Big)^D,
\end{equation}
where $L$ can be $\sqrt{3 \pi e / 2}$. But,
using~\eqref{eq:gamma_bracket_dudley} together with the fact that
$N^{[]}(A, \epsilon / 2) \leq N(A, \epsilon, d_\infty)$, we obtain
\begin{equation}
  \label{eq:gamma_brack_finite_dim}
  \gamma^{[]}(B(\delta) ) \leq \sqrt{D} \int_0^\delta \sqrt{ \ln
    \Big( \frac{2 L r(\bar A) \delta}{\varepsilon} \Big) } \mathrm d
  \epsilon \leq L \delta \sqrt{D (\ln r(\bar A) + 1)}.
\end{equation}
So, we have the control required in Theorem~\ref{thm:ERM_bound}: $
\gamma^{[]}(B(\sqrt \delta)) \leq L \varphi(\delta)$, with
\begin{equation*}
  \varphi(\delta) = \sqrt \delta \sqrt{D (\ln r(\bar A) + 1)},
\end{equation*}
which is a function that satisfies the assumptions of
Theorem~\ref{thm:ERM_bound}. Note that
\begin{equation}
  \label{eq:phiminusonestar}
  \varphi^{-1 *}(x) = \frac{x^2 D \ln ( r(\bar A) + 1)}{4},
\end{equation}
so we have the following.
\begin{corollary}
  \label{cor:erm_finite_dim}
  Grant Assumptions~\ref{ass:support} and \ref{ass:countable}\textup,
  and assume that $A \subset \bar A$\textup, where $\bar A$ is a
  linear space with finite dimension $D$. Then\textup, if $\bar
  \alpha_n$ is a $\rho$-ERM according to
  Definition~\ref{def:erm}\textup, we have for any $\epsilon > 0$ and
  $x > 0$\textup:
  \begin{equation*}
    P(\ell_{\bar \alpha_n}) \leq P(\ell_{\alpha_*}) + (1 + \epsilon)
    \rho + \frac{c (1 + \epsilon)^2 \ln ( r(\bar A) +
      1)}{\epsilon} \frac{D}{n} (1 + x \vee \sqrt x)^2
  \end{equation*}
  with a probability larger than $1 - L e^{-x}$\textup, where $c =
  c_{b, \norm{\alpha_0}_\infty}$. In particular\textup, we obtain
  \begin{equation}
    \label{eq:ermbound1}
    \E^n \norm{\bar \alpha_n - \alpha_0}^2  \leq 2 \rho +
    \inf_{\alpha \in A} \norm{\alpha - \alpha_0}^2 +  c \ln (
    r(\bar A) + 1) \frac{D}{n}.
  \end{equation}
\end{corollary}

\begin{proof}
  Note that Assumption~\ref{ass:margin} is met since $\bar A$ is
  linear. So, Theorem~\ref{thm:ERM_bound} together
  with~\eqref{eq:phiminusonestar} gives the first inequality. The
  second inequality follows by choosing $\epsilon = 1$, by subtracting
  $P(\ell_{\alpha_0})$ at both sides of the inequality, and by
  integration with respect to~$x$.
\end{proof}

The next step is, usually, to have a control on the
\emph{approximation} or \emph{bias} term $\inf_{\alpha \in A}
\norm{\alpha - \alpha_0}^2$, and to choose a sieve with a dimension
that equilibrates the bias term with the ``variance'' term $D / n$,
hence the name \emph{bias-variance} problem, see \cite{cucker_smale02}
for instance. Usually, this is done using the assumption that
$\alpha_0$ belongs to some smoothness class of functions, together
with some results from approximation theory. This is where the problem
of adaptive estimation arises: the choice of the optimal $D$ depends
on the parameters of the smoothness class itself, which is unknown in
practice. So, one has to find a procedure with the capability to
select automatically a sieve or a \emph{model} among a collection $\{
A_m : m \in \mathcal M \}$. This is usually done using
model-selection, see the seminal paper \cite{bbm99}. Model selection
in the setup considered here has been studied in \cite{CGG}. In
Section~\ref{sec:learning} below, we consider an alternative approach,
based on a popular aggregation procedure. It will allow the
construction of smoothness and structure adaptive estimators, see
Section~\ref{sec:adaptive_estimation}.

\section{Agnostic learning, aggregation}
\label{sec:learning}

Let $A = A(\Lambda) := \{ \alpha_\lambda : \lambda \in \Lambda \}$ be
a set of arbitrary functions called \emph{dictionary} with cardinality
$M$. For instance, this can be a set of so-called \emph{weak}
estimators, computed based on a set of observations independent of the
sample $D_n$. We consider the problem of agnostic learning: without
any assumption on $\alpha_0$, excepted for some boundedness
assumption, we want to construct (from the data) a procedure $\hat
{\alpha}_n$ with a risk as close as possible to the smallest risk over
$A$. Namely, we want to obtain an oracle inequality of the form
\begin{equation*}
  \E^n \| \hat {\alpha}_n - \alpha_0 \|^2 \leq c \min_{\alpha \in A}
  \|\alpha - \alpha_0 \|^2 + \phi(n, M),
\end{equation*}
where $c \geq 1$ and $\phi(n, M)$ is called the {\it residue} or
\emph{rate of aggregation}, which is a quantity that we want to be
small as $n$ increases. An oracle inequality that holds with $c = 1$
is called \emph{sharp}.

This problem has been considered in several statistical models, mainly
in regression, density and classification, see among others
\cite{nemirovski00, catbook:01, jrt:06, leung_barron06,
  dalalyan_tsybakov07, yang:00, audibert-2009-37}. For instance, we
know from \cite{tsy:03} that the optimal rate of aggregation in the
Gaussian regression model is $\phi(n, M) = (\log M) /n$ (in the sharp
oracle inequality context). This rate is achieved by the algorithms of
aggregation with cumulative exponential weights, see \cite{jrt:06,
  audibert-2009-37} and aggregation with exponential weights, see
\cite{dalalyan_tsybakov07} (when the error of estimation is measured
by the empirical norm, a similar result for the integrated norm is, as
far as we know, still a conjecture).

Aggregation with exponential weights is a popular algorithm. It is of
importance in machine learning, for estimation, prediction using
expert advice, in PAC-Bayesian learning and other settings, see
\cite{MR2409394}, \cite{audibert-2009-37} and \cite{catbook:01}, among
others. However, there is no result for this algorithm in the general
model~\eqref{eq:model}, nor for any of the particular cases given in
Section~\eqref{sec:model_examples}. In this Section, we construct this
algorithm for model~\eqref{eq:model}, and give in
Theorem~\ref{thm:oracle} below an oracle inequality.

The idea of aggregation is to mix the elements from $A(\Lambda)$:
using the data, compute weights $\theta(\alpha) \in [0,1]$ for each
$\alpha \in A(\Lambda)$ satisfying $\sum_{\lambda \in \Lambda}
\theta(\alpha_{\lambda}) = 1$. These weights give a level of
significance to $\alpha$. The aggregate is the convex combination
\begin{equation}
  \label{eq:aggregate}
  \hat \alpha_n := \sum_{\lambda \in \Lambda} \theta(\alpha_\lambda)
  \alpha_\lambda,
\end{equation}
where the weight of $\alpha \in A(\Lambda)$ is given by
\begin{equation}
  \label{eq:weights}
  \theta(\alpha) := \frac{\exp\big( - n P_{n}(\ell_\alpha) / T
    \big) }{ \sum_{\lambda \in \Lambda} \exp \big( -
    n P_{n}(\ell_{\alpha_\lambda}) / T \big) },
\end{equation}
where $T > 0$ is the so-called \emph{temperature} parameter and where
we recall that
\begin{equation*}
  % \label{eq:empirial_risk_learning}
  P_n(\ell_\alpha) = \frac{1}{n} \sum_{i=1}^n \int_0^1 \alpha(t,
  X_i)^2 Y^i(t) dt - \frac{2}{n} \sum_{i=1}^n \int_0^1 \alpha(t,
  X_i) d N^i(t)
\end{equation*}
is the empirical risk of $\alpha$. The shape of this mixing estimator
is easily explained. Indeed, the weighting scheme~\eqref{eq:weights}
is the only minimizer of
\begin{equation}
  \label{eq:pena_linearized_risk}
  \mathsf R_{n}(\theta) + \frac{T}{n} \sum_{\lambda \in
    \Lambda} \theta_\lambda \log \theta_\lambda
\end{equation}
among all $\theta \in \Theta$ (we use the convention $0 \log 0 = 0$)
where
\begin{equation*}
  \Theta := \big\{ \theta \in \mathbb R^{M} : \theta_\lambda
  \geq 0,\; \sum_{\lambda \in \Lambda} \theta_\lambda = 1 \big\},
\end{equation*}  
and where $\mathsf R_{n}(\theta)$ is the \emph{linearized empirical
  risk}
\begin{equation*}
  % \label{eq:linearized_empirical_risk}
  \mathsf R_{n}(\theta) := \sum_{\lambda \in \Lambda}
  \theta_\lambda P_{n}(\ell_{\alpha_\lambda}).
\end{equation*}
Equation~\eqref{eq:pena_linearized_risk} is the linearized risk of
$\theta \in \Theta$, which is penalized by a quantity proportional to
the Shannon's entropy of $\theta$. The resulting aggregated estimator
$\hat \alpha_n$ is then something between the ERM among the elements
of $A(\Lambda)$ (when $T$ is small), and the mean of the elements of
$A(\Lambda)$ (when $T$ is large).

\begin{theorem}
  \label{thm:oracle}
  Assume that $\norm{\alpha_0}_\infty < +\infty$, and that there is $b
  > 0$ such that $\norm{\alpha}_\infty \leq b$ for any $\alpha \in
  A(\Lambda)$. Then, for any $\epsilon > 0$, the mixing estimator
  $\hat \alpha_n$ defined by~\eqref{eq:aggregate} satisfies
  \begin{equation*}
    \E^n \norm{\hat \alpha_n - \alpha_0}^2 \leq (1+\epsilon)
    \inf_{\lambda \in \Lambda} \norm{\alpha_\lambda - \alpha_0}^2
    + \frac{c \log M}{n}
  \end{equation*}
  for any $n \geq 1$, where $c = c_{b, \norm{\alpha_0}_\infty, T,
    \epsilon}$.
  % The following values work\textup: $a = 1$
  % and $C(a, Q) = ?$, $a = 1/2$ and $C(a, Q) = ?$ and $a = 1/4$ and
  % $C(a, Q) = ?$.
  % \begin{equation*}
  %   C(a) = (1+a) \Big( \frac{1}{T} + 6 Q \big( \frac{(1 Q + 4)(1 +
  %     a)}{a} + \frac{1 Q + 2 (C_N + 1 Q)}{3} \big) \Big).
  % \end{equation*}
\end{theorem}

Theorem~\ref{thm:oracle} is a model-selection type oracle inequality
for the aggregation procedure given by~\eqref{eq:aggregate}. The
residual term in the oracle inequality is of order $(\log M) / n$,
which is the correct rate of convex aggregation, see \cite{tsy:03} (in
the Gaussian regression setup, and for other models with margin
parameter equal to $1$, see \cite{these_bobby}).

\begin{remark}
  The main criticism one can make about Theorem~\ref{thm:oracle} is
  that it is not sharp: the leading constant is $1 + \epsilon$ instead
  of $1$ in front of $\inf_{\lambda \in \Lambda} \norm{\alpha_\lambda
    - \alpha_0}^2$, and the constant $c$ in front of the residue is
  far from being optimal. The consequence is that we are not able in
  this setting to give a theoretically optimal value for $T$.  Sharp
  oracle inequalities are available for aggregation with exponential
  weights or cumulative weights, see \cite{dalalyan_tsybakov07},
  \cite{jrt:06} and \cite{audibert-2009-37}, see also references
  mentioned above.  However, in the setup considered here, the proof
  of a sharp oracle inequality seems quite challenging, and will be
  the subject of further investigations.
  % Hence $T$ is a regularization parameter, and its choice is a
  % question that deserves specific developments in the problem
  % considered here (theoretical choices of $T$ are given in
  % \cite{jrt:06}, see also \cite{audibert-2009-37}).
\end{remark}

\section{Structure and smoothness adaptive estimation}
\label{sec:adaptive_estimation}

In this Section, we propose an application of the results obtained in
Sections~\ref{sec:erm} and~\ref{sec:learning}. We construct an
estimator that adapts to the smoothness of $\alpha_0$ in a purely
nonparametric setting, see Section~\ref{sec:purely_nonparametric}, and
to its structure in a single-index setup, see
Section~\ref{sec:dimension_reduction}. The steps of the construction
of the estimator are given in Definition~\ref{def:steps} below. As
usual with algorithms coming from statistical learning theory, we need
to split the sample (a very particular exception can be found
in~\cite{leung_barron06}). To simplify, we shall assume that the
sample size is $2n$, see~\eqref{eq:whole_sample}, so $D_{2n}$ is the
full sample.

\begin{definition}
  \label{def:steps}
  The steps for the computation of an aggregated estimator $\hat
  \alpha_n$ are the following:
  \begin{enumerate}
  \item split the whole sample $D_{2n}$ (see \eqref{eq:whole_sample})
    into a training sample $D_{n, 1}$ of size $n$ and a \emph{learning
      sample} $D_{n, 2}$ of size $n$;
  \item choose a collection of sieves $\{ A_m : m \in \mathcal M_n \}$
    and compute, using $D_{n, 1}$, the corresponding empirical risk
    minimizers $\{ \bar \alpha_m : m \in \mathcal M_n \}$ (see
    Definition \ref{def:erm});
  \item using the learning sample $D_{n, 2}$, compute the aggregated
    estimator $\hat \alpha_n$ based on the dictionary $\{ \bar
    \alpha_m : m \in \mathcal M_n \}$, see~\eqref{eq:aggregate}
    and~\eqref{eq:weights}.
  \end{enumerate}
\end{definition}

Examples of collections $\{ A_m : m \in \mathcal M_n \}$ are given in
Appendix~\ref{sec:approximation}, together with the necessary control
of the $\L^\infty$-index (see Section~\ref{sec:A_finite_dimension}),
and a useful approximation result.

\begin{remark}[Jackknife]
  The behavior of the aggregate $\hat {\mathsf \alpha}_n$ typically
  depends on the split selected in Step~1, in particular when the
  number of observations is small. Hence, a good strategy is to
  jackknife: repeat, say, $J$ times Steps 1--3 to obtain aggregates
  $\{ \hat {\mathsf \alpha}_n^{(1)}, \ldots, \hat {\alpha}_n^{(J)}
  \}$, and compute the mean:
  \begin{equation*}
    \hat {\alpha}_n := \frac{1}{J} \sum_{j=1}^J \hat {\alpha}_n^{(j)}.
  \end{equation*}
  This jackknifed estimator should provide more stable results than a
  single aggregate. Moreover, by convexity of the risk $\alpha \mapsto
  P(\ell_\alpha)$, the jackknifed estimator satisfies the same risk
  bounds as a single aggregate.
\end{remark}

\subsection{Adaptive estimation in the purely nonparametric setting}
\label{sec:purely_nonparametric}

In model~\eqref{eq:model}, the behaviour of $\alpha_0(t, x)$ with
respect to time $t$ and with respect to the covariates $x$ have no
statistical reason to be linked. So, in a purely nonparametric
setting, it is mandatory to consider anisotropic smoothness for
$\alpha_0$. We shall assume in the statement of the upper bound, see
Theorem~\ref{thm:adaptive_nonparametric} below, that $\alpha_0 \in
B_{2, \infty}^{\bs s}$, where $\alpha_0 \in B_{2,\infty}^{\bs s}$ is
an anisotropic Besov space (see Appendix~\ref{sec:approximation}) and
$\bs s = (s_1, \ldots, s_{d+1})$ is a vector of smoothness, where
$s_i$ is the smoothness in the $i$th coordinate. For the construction
of the adaptive estimator, see Step~2 above, we need a collection of
sieves $\{ A_m : m \in \mathcal M_n\}$.
\begin{definition}[Collection]
  \label{def:collection}
  We take $\{ A_m' : m \in \mathcal M_n\}$ as:
  \begin{itemize}
  \item a collection of linear spaces spanned by piecewise polynomials
    (see Section~\ref{sec:dyadic_polynomials}), with degrees not
    larger than $l_i$ in the $i$th coordinate, or
  \item a collection of linear spaces spanned by wavelets (see
    Section~\ref{sec:wavelets}) with $l_i$ vanishing moments in the
    $i$th coordinate.
  \end{itemize}
  In both cases, we say that $(l_1, \ldots, l_{d+1})$ is the
  \emph{smoothness} of the collection, and we take
  \begin{equation*}
    \mathcal M_n := \{ (m_1, \ldots, m_{d+1}) \in \mathbb N^{d+1} :
    2^{m_i} \leq n^{1 / (d+1)} \text{ for } i = 1, \ldots, d+1  \}.
  \end{equation*}
  Finally, we fix a constant $b > 0$ and take $A_m := \{ \alpha \in
  A_m' : \norm{\alpha}_\infty \leq b \}$ for every $m \in \mathcal
  M_n$.
\end{definition}

For the statement of the adaptive upper bound, we need the following
assumption, which is a stronger version of the previous
Assumption~\ref{ass:support}.
\begin{assumption}
  \label{ass:design}
  Assume that $P_X$ has a density $f_X$ with respect to the Lebesgue
  measure, which is bounded and with support $[0, 1]^{d+1}$. Moreover,
  we assume that $\norm{\alpha_0}_\infty \leq b$, where $b$ is known
  (it is used in the definition of the sieves, see
  Definition~\ref{def:collection}).
\end{assumption}

Now, we can use together Corollary \ref{cor:erm_finite_dim} (see
Section \ref{sec:A_finite_dimension}), Theorem~\ref{thm:oracle} (see
Section~\ref{sec:learning}) and Lemma~\ref{lem:approximation} (see
Appendix~\ref{sec:approximation}) to derive an adaptive upper bound.
Take $\rho = 1/n$ in Corollary~\ref{cor:erm_finite_dim} and, say,
$\epsilon = 1$ in Theorem~\ref{thm:oracle}, to obtain
\begin{equation*}
  \E^{2n} \norm{\hat \alpha_n - \alpha_0}^2 \leq 2 \inf_{m \in
    \mathcal M_n} \Big( \inf_{\alpha \in A_m} \norm{\alpha -
    \alpha_0}^2 + \frac{c_b D_m}{n} \Big) + c_{b, T} \frac{\log |\mathcal
    M_n|}{n},
\end{equation*}
where for $m = (m_1, \ldots, m_{d+1}) \in \mathcal M_n$,
\begin{equation*}
  D_m := \prod_{j=1}^{d+1} D_{m_j} = \prod_{j=1}^{d+1} 2^{m_j}.
\end{equation*}
Let us assume that $\alpha_0 \in B_{2, \infty}^{\bs s}$, where $\bs s
= (s_1, \ldots, s_{d+1})$ satisfies $s_i > (d + 1) / 2$ for each $i =
1, \ldots, d+1$. Assumption~\ref{ass:design} entails $\norm{\alpha}^2
\leq \norm{f_X}_\infty \norm{\alpha}_2^2$ for any $\alpha \in \L^2[0,
1]^{d+1}$, where $\norm{\alpha}_2^2 = \int_{[0, 1]} \int_{[0, 1]^d}
\alpha(t, x)^2 dt dx$. So, using Lemma~\ref{lem:approximation}, we
have when $\alpha_0 \in B_{2, \infty}^{\bs s}$:
\begin{equation*}
  \E^{2n} \norm{\hat \alpha_n - \alpha_0}^2 \leq c
  \Big(\sum_{j=1}^{d+1} D_{m_j}^{-2 s_j} +
  \frac{\prod_{j=1}^{d+1} D_{m_j}}{n} + \frac{\log |\mathcal M_n|}{n}
  \Big),
\end{equation*}
where $c = c_{{b, T, \bs s, d, \norm{f_X}_\infty}, |\alpha_0|_{B_{2,
      \infty}^{\bs s}}}$. Note that $(\log |\mathcal M_n|) / n \leq
c_{d} (\log n) / n$, so the rate of convergence is given by the
optimal tradeoff between the bias and the variance terms. Since $s_i >
(d+1) / 2$ for any $i = 1, \ldots, d+1$, we have $n^{\bar {\bs s} /
  s_i(2 \bar {\bs s} + d + 1)} \leq n^{1 / (d+1)}$, so we can choose
$m = (m_1, \ldots, m_{d+1}) \in \mathcal M_n$ such that
\begin{equation}
  \label{eq:oracle_dimension}
  2^{m_i - 1}  \leq n^{\frac{\bar {\bs s} / s_i}{2 \bar {\bs
        s} + d + 1}} \leq 2^{m_i} \text{ for } i = 1, \ldots, d+1.
\end{equation}
This gives
\begin{equation*}
  \sum_{j=1}^{d+1} D_{m_j}^{-2 s_j} + \frac{\prod_{j=1}^{d+1}
    D_{m_j}}{n} = c_d n^{-2 \bar {\bs s} / (2 \bar {\bs s} + d + 1)},
\end{equation*}
so we proved the following theorem. 
\begin{theorem}
  \label{thm:adaptive_nonparametric}
  Grant Assumption~\ref{ass:design}, and consider a collection $\{ A_m
  : m \in \mathcal M_n\}$ given by Definition~\ref{def:collection}
  with smoothness $(l_1, \ldots, l_{d+1})$. Assume that $\alpha_0 \in
  B_{2,\infty}^{\bs s}$, where $\bs s = (s_1, \ldots, s_{d+1})$
  satisfies $(d+1) / 2 < s_i \leq l_i$ for each $i = 1, \ldots, d+1$.
  Then, if $\hat \alpha_n$ is the aggregated estimator given by
  Steps~1-3, we have
  \begin{equation*}
    \E^{2n} \norm{\hat \alpha_n - \alpha_0}^2 \leq c n^{-2 \bar {\bs s} /
      (2 \bar {\bs s} + d + 1)},
  \end{equation*}
  where $c = c_{{b, T, \bs s, d, \norm{f_X}_\infty}, |\alpha_0|_{B_{2,
        \infty}^{\bs s}}}$.
\end{theorem}
The rate $n^{-2 \bar {\bs s} / (2 \bar {\bs s} + d + 1)}$ is the
optimal rate of convergence (in the minimax sense) in this model,
under the extra assumption that $f_X$ is bounded away from zero on
$[0, 1]^d$, see Theorem~3 in \cite{CGG}. Hence,
Theorem~\ref{thm:adaptive_nonparametric} shows that $\hat \alpha_n$
adapts to the smoothness of $\alpha_0$ over a range of Besov spaces
$B_{2,\infty}^{\bs s}$, for $(d+1) / 2 < s_i \leq l_i$.

\subsection{Dimension reduction, single-index}
\label{sec:dimension_reduction}

The mark $X$ is $d$-dimensional so the intensity $\alpha_0$ takes $d +
1$ variables. As with any other nonparametric estimation model, we
know that when $d$ gets large the dimension has a significant impact
on the accuracy of estimation. This the so-called \emph{curse of
  dimensionality} phenomenon, which is reflected by the rate $n^{-2
  \bar {\bs s} / (2 \bar {\bs s} + d + 1)}$, see
Theorem~\ref{thm:adaptive_nonparametric} above. This rate is slow if
$d$ is large compared to $\bar {\bs s}$. In this Section, we propose a
way to ``get back'' the rate $n^{-2 \bar {\bs s} / (2 \bar {\bs s} +
  2)}$, using single-index modelling. Thanks to our approach based on
aggregation, we are able to construct an estimator that automatically
takes advantage (without any prior testing) of the single-index
structure when possible: the rate is then $n^{-2 \bar {\bs s} / (2
  \bar {\bs s} + 2)}$, otherwise it is the the purely nonparametric
rate $n^{-2 \bar {\bs s} / (2 \bar {\bs s} + d + 1)}$. This idea of
mixing nonparametric and semiparametric estimators can be also found
in~\cite{yang:00} for density estimation.

Dimension reduction techniques usually involves an assumption on the
structure of the object to be estimated. Main examples are the
additive and the single-index models. Additive modelling was proposed
by \cite{LNVdG} in the same context as the one considered here, with
very different techniques (kernel estimation and back-fitting). In this
paper, we focus on single-index modelling (see
Remark~\ref{rem:cox-aalen} below). On single-index models (mainly in
regression) and the corresponding estimation problems (estimation of
the link function, estimation of the index), see \cite{spok01},
\cite{delecroix_etal03}, \cite{hardle_xia06}, \cite{delecroix_etal06},
\cite{geenens_delecroix05}, \cite{gaiffas_lecue07}, \cite{MR2438819}
among many others. The single-index structure is as follows: assume
that there is an unknown function $\beta_0 : \mathbb R_+ \times
\mathbb R \rightarrow \mathbb R_+$ (called \emph{link function}, with
has unknown smoothness here) and an unknown vector $v_0 \in \mathbb
R^d$ (called \emph{index}) such that
\begin{equation}
  \label{eq:SImodel}
  \alpha_0(t, x) = \beta_0(t, v_0^{\top} x).
\end{equation}
In order to make the representation~\eqref{eq:SImodel} unique
(identifiability), we shall assume (see Assumption~\ref{ass:sim}
below) that $v_0 \in S_+^{d-1}$, where $S_+^{d-1}$ is the half-unit
sphere defined by
\begin{equation}
  \label{eq:half_unit_sphere}
  S_+^{d-1} = \big\{ v \in \mathbb R ^d : |v|_2 = 1 \text{
    and } v_d \geq 0 \big\},
\end{equation}
where $|\cdot|_2$ is the Euclidean norm over $\mathbb R^d$;

The steps of the construction of the adaptive estimator in this
context follows the ones from Definition~\ref{def:steps}, but the
dictionary is enlarged by a set $\{ \bar \alpha_{m, v}^{\rm SIM} : m
\in \mathcal M_n^{\rm SIM}, v \in S_\Delta^{d-1} \}$, of empirical
risk minimizers, where $S_\Delta^{d-1}$ is a $\Delta$-net of
$S_+^{d-1}$. So, compared to Section~\ref{sec:purely_nonparametric},
the idea is simply to add estimators that works under the single-index
assumption in the dictionary.
\begin{definition}
  \label{def:stepssim}
  The steps for the computation of the aggregated estimator $\hat
  \alpha_n$ are the following:
  \begin{enumerate}
  \item split the whole sample $D_{2n}$ (see \eqref{eq:whole_sample})
    into a training sample $D_{n, 1}$ of size $n$ and a \emph{learning
      sample} $D_{n, 2}$ of size $n$;
  \item Compute a $\Delta = (n \log n)^{-1/2}$-net of the half-unit
    sphere $S_+^{d-1}$ denoted by $S_\Delta^{d-1}$ and for each $v \in
    S_\Delta^{d-1}$ compute the pseudo-training samples
    \begin{equation}
      \label{eq:SIM_proj_training}
      D_{n, 1}(v) := \big[ (v^\top X_i, N^i(t), Y^i(t)) : t \in [0,
      1], 1 \leq i \leq n\big],
    \end{equation}
    where the $d$-dimensional marks $X_i$ are simply replaced
    univariate marks $v^\top X_i$.
  \item Fix a collection of $2$-dimensional sieves ($d=1$) $\{
    A_m^{\rm SIM} : m \in \mathcal M_n^{\rm SIM} \}$ given by
    Definition~\ref{def:collection}. Compute, for every $m \in
    \mathcal M_n^{\rm SIM}$ and $v \in S_\Delta^{d-1}$, empirical risk
    minimizers $\bar \beta_{m, v}^{\rm SIM}$, over $A_m^{\rm SIM}$, of
    the empirical risks
    \begin{align}
      \nonumber P_{n, 1}^{(v)}(\ell_{\alpha}) = \frac{1}{n}
      \sum_{i=1}^n \int_0^1 \alpha(t, v^\top X_i)^2 Y^i(t) dt -
      \frac{2}{n} \sum_{i=1}^n \int_0^1 \alpha(t, v^\top X_i) d
      N^i(t),
    \end{align}
    and define
    \begin{equation*}
      \bar \alpha_{m, v}^{\rm SIM}(\cdot, \cdot) := \bar \beta_{m,
        v}^{\rm SIM}(\cdot, v^\top \cdot).
    \end{equation*}
    (so that each $\bar \alpha_{m, v}^{\rm SIM}$ works as if $v$ were
    the true index).
  \item follow Steps~2 and~3 from Definition~\ref{def:steps}, where we
    add the estimators $\{ \bar \beta_{m, v}^{\rm SIM} : m \in
    \mathcal M_n^{\rm SIM}, v \in S_\Delta^{d-1} \}$ to the set of
    purely nonparametric estimators $\{ \bar \alpha_m : m \in \mathcal
    M_n \}$ in the aggregation step.
  \end{enumerate}
\end{definition}
An important point of this algorithm is that we do not estimate the
index directly: we mix estimators in order to \emph{adapt} to the
unknown $v_0$ and to the unknown smoothness of $\beta_0$. This
approach was previously adopted in \cite{gaiffas_lecue07} for the
estimation of the regression function. Note that the size of
$S_{\Delta}^+$ increases strongly with $n$ and $d$, so this method is
restricted to a reasonably small $d$. High dimensional covariates
cannot be handled in such a semiparametric approach, this problem will
be the subject of another work. About high dimension,
see~\cite{tibshirani97}, where the LASSO has been studied in the Cox
model.

The following set of assumptions gives the identifiability of
model~\eqref{eq:SImodel} (see for instance the survey paper by
\cite{geenens_delecroix05}, or Chapter~2 in~\cite{horowitz98}),
excepted for~\eqref{eq:link_lip} and~\eqref{eq:link_bounded_below}
which are technical assumptions.
\begin{assumption}
  \label{ass:sim}
  Assume that~\eqref{eq:SImodel} holds, and that
  \begin{itemize}
  \item $x \mapsto \beta_0(t, x)$ is not constant over the support of
    $v_0\T X$;
  \item $X$ admits at least one continuously distributed coordinate
    (w.r.t. the Lebesgue measure);
  \item the support of $X$ is not contained in any linear subspace of
    $\mathbb R^d$;
  \item $v_0 \in S_+^{d-1}$;
  \item there is $c_0 > 0$ such that for any $x, y \in [0, 1]^d$, any
    $t \geq 0$:
    \begin{equation}
      \label{eq:link_lip}
      | \beta_0(t, x) - \beta_0(t, y) | \leq c_0 |x - y|;
    \end{equation}
  \item there is $b_0 > 0$ such that
    \begin{equation}
      \label{eq:link_bounded_below}
      \inf_{(t, x) \in [0, 1]^{d+1}} \beta_0(t, x) \geq b_0.
    \end{equation}
  \end{itemize}
\end{assumption}

\begin{remark}
  \label{rem:cox-aalen}
  In the problem of estimating the intensity of a counting process in
  presence of covariates, two of the most popular models are special
  cases of the single-index model, as described in
  Equation~\eqref{eq:SImodel}:
  \begin{itemize}
  \item the Cox model (see \cite{Cox}), where there exists an unknown
    function $\beta_0$ such that:
    \begin{equation}
      \alpha_0(t, x) = \beta_0(t)\exp( v_0^{\top} x).
    \end{equation}
    and 
  \item the Aalen model (see \cite{aalen80}), which can be written as:
    \begin{equation}
      \alpha_0(t, x) = \beta_0(t)+ v_0^{\top} x.
    \end{equation}
  \end{itemize}
  This emphasizes the relevance of considering single-index models in
  this context, and the use of anisotropic smoothness. This paper is
  only a first step in this direction, for the expected rate of
  convergence in these two models would be $n^{-2s / (2s + 1)}$ when
  the link function has smoothness $s$ in some sense. Adaptive
  estimation by aggregation, including the Cox and Aalen models, will
  be addressed in a forthcoming paper.
\end{remark}

\begin{theorem}
  \label{thm:adaptivesim}
  Grant the same assumptions as in
  Theorem~\ref{thm:adaptive_nonparametric} and let $\hat \alpha_n$ be
  the aggregated estimator from Definition~\ref{def:stepssim}.
  \begin{itemize}
  \item If Assumption~\ref{ass:sim} holds (single-index) with $\beta_0
    \in B_{2,\infty}^{\bs s}$, where $\bs s = (s_1, s_2)$ satisfies $1
    < s_i \leq l_i$ for $i = 1, 2$, we have
    \begin{equation*}
      \E^{2n} \norm{\hat \alpha_n - \alpha_0}^2 \leq c n^{-2 \bar {\bs s} /
        (2 \bar {\bs s} + 2)}
    \end{equation*}
    for $n$ large enough, where $c = c_{{b, T, \bs s, d,
        \norm{f_X}_\infty}, |\beta_0|_{B_{2, \infty}^{\bs s}}, v_0,
      b_0, c_0}$.
  \item Otherwise, we have, when $\alpha_0 \in B_{2,\infty}^{\bs s}$,
    where $\bs s = (s_1, \ldots, s_{d+1})$ satisfies $(d+1) / 2 < s_i
    \leq l_i$ for each $i = 1, \ldots, d+1$, that
    \begin{equation*}
      \E^{2n} \norm{\hat \alpha_n - \alpha_0}^2 \leq c n^{-2 \bar {\bs
          s} / (2 \bar {\bs s} + d + 1)},
    \end{equation*}
    for $n$ large enough, where $c = c_{{b, T, \bs s, d,
        \norm{f_X}_\infty}, |\alpha_0|_{B_{2, \infty}^{\bs s}}}$.
  \end{itemize}
\end{theorem}

The proof of this theorem is given in Section~\ref{sec:main_proofs}.
This theorem proves that $\hat \alpha_n$ adapts to the smoothness of
the intensity, and to its structure: if the single-index
model~\eqref{eq:SImodel} holds, then the rate is $n^{-2 \bar {\bs s} /
  (2 \bar {\bs s} + 2)}$, which is the optimal rate when $X$ is
one-dimensional. Otherwise, the rate of convergence is $n^{-2 \bar
  {\bs s} / (2 \bar {\bs s} + d + 1)}$ when the covariate is
$d$-dimensional. Of course, this result is not surprising, since any
kind of estimator can be used in the dictionary to be
aggregated. However, note that the proof of
Theorem~\ref{thm:adaptivesim} involves a technical tool concerning
counting processes, namely a concentration inequality for the
likelihood ratio between two indexes in $S_+^{d-1}$, see
Lemma~\ref{lem:devialikeli} in Section~\ref{sec:main_proofs}.

%%% MODIFIER LA PREUVE DU THEOREME 6, EN RAJOUTANT L'ARGUMENT POUR QUE
%%% la fonction de lien loi Lipshitz

% \begin{remark}
%   Note~\eqref{eq:link_lip}) is not restrictive: it is a consequence of
%   the fact that $\beta_0 \in B_{2, \infty}^{\bs s} \subset B_{\infty,
%     \infty}^{(1, 1)}$ when $s_1, s_2 > 1$, where one can choose $c_0 =
%   |\beta_0|_{B_{\infty, \infty}^{1, 1}}$.

%   For instance, if $\mathcal B_s = B_{p, \infty}^s$, we can
%   derive~\eqref{eq:link_ass1} when $p \geq 2$ from the embedding
%   $B_{p, \infty}^s \subset B_{p, \infty}^1$ (since $s > 1$) by
%   choosing , see Chapter~2 in \cite{devore_lorentz93}.

%   (the Besov space on $\mathbb R$), then it is easy to see
%   that~\eqref{eq:link_ass1} holds with $L = |\beta_0|$

%   \texttt{ON DOIT POUVOIR ENLEVER L'HYPOTHESE...}

% \end{remark}

% \begin{remark}
%   Assumption~\eqref{eq:link_ass1} is satisfyied for any $s_{\rm min} <
%   s$ when $\beta_0$ belongs to the Besov space $B_{p,\infty}^s$ for
%   instance, where $L = $ \texttt{attention ici !!!
%   }. Assumption~\eqref{eq:link_ass2} is not restrictive, for instance
%   in the particular model of survival times, see
%   Section~\ref{sec:censored_data}, such an assumption is natural from
%   the practical point of view, see for instance ????
% \end{remark}

\section{Proofs}
\label{sec:main_proofs}

\subsection*{Proof of Proposition~\ref{prop:deviation_martingale}}

\begin{proof}[Proof of Proposition~\ref{prop:deviation_martingale}]  
  Let us define the process
  \begin{equation*}
    Z_n(\alpha, t) := \frac{1}{\sqrt{n}} \sum_{i=1}^n \int_0^t
    \alpha(u, X_i) d M^i(u):=\sum_{i=1}^n  Z^i_n(\alpha, t),
  \end{equation*}
  so that $Z_n(\alpha) = Z_n(\alpha, 1)$. The predictable variation of
  $M^i$ is given by $\langle M^i(t) \rangle = \int_0^t \alpha_0(u,
  X_i) Y^i(u) du$, so we have
  \begin{equation*}
    \langle Z^i_n(\alpha, t) \rangle = \frac{1}{n}
    \int_0^t \alpha(u, X_i)^2 \alpha_0(u, X_i) Y^i(u) du
  \end{equation*}
  for any $t \in [0, 1]$. Moreover, we have $\Delta M^i(t) \in \{ 0, 1
  \}$ for any $i=1, \dots,n$ since the counting processes $N^i$ have
  an intensity. We can write $ Z^i_n = Z^{i,c}_n + Z^{i,d}_n$ where
  $Z^{i,c}_n$ is a continuous martingale and where $Z^{i,d}_n$ is a
  purely discrete martingale (see e.g. \cite{lipstershiryayev}). Let
  $h > 0$ be fixed and define $U_h^i(t) := h Z^i_n(\alpha,t) -
  S_h^i(t)$, where $S_h^i(t)$ is the compensator of
  \begin{equation}
    \label{eq:S_h_def}
    \frac12 h^2 \langle Z^{i,c}_n(\alpha, t) \rangle + \sum_{s \leq t}
    \left(\exp(h | \Delta Z^i_n(\alpha, s)|) -1 -h| \Delta Z^i_n(\alpha,
      s)| \right).
  \end{equation}
  We know from the proof of Lemma 2.2 and Corollary 2.3 of
  \cite{vandegeer95}, see also~\cite{lipstershiryayev}, that
  $\exp(U_h^i(t))$ is a super-martingale. Then, if $S_h := \sum_{i=1}^n
  S_h^i$, $U_h := \sum_{i=1}^n U_h^i$, we have
  \begin{align}
    \nonumber \E^n [ e^{h Z_n(\alpha)} \ind{\langle Z_n(\alpha)
      \rangle \leq \delta^2} ] &\leq \big( \E^n[ e^{2 U_h(1) } ]
    \big)^{1/2} \big( \E^n [ e^{2 S_h(1)} \ind{\langle Z_n(\alpha)
      \rangle \leq \delta^2} ]
    \big)^{1/2} \\
    &\leq \big( \E^n [ e^{2 S_h(1)} \ind{\langle Z_n(\alpha) \rangle
      \leq \delta^2} ] \big)^{1/2}. \label{eq:deviamart1}
  \end{align}
  The last inequality holds since $\exp(U_h^i(t))=\exp( h
  Z^i_n(\alpha,t) - S_h^i(t))$ are independent super-martingales with
  $U_h^i(0)=0$, so that $\E [ \exp(2 U_h^i(t))] \leq 1$, for $i=1,
  \dots, n$. Let us decompose $M^i= M^{i,c}+ M^{i,d}$, with $M^{i,c}$
  a continuous martingale and $M^{i,d}$ a purely discrete
  martingale. The process $V_2^i(t) := \langle M^i(t) \rangle$ is the
  compensator of the quadratic variation process $[M^i(t)]= \langle
  M^{i,c}(t) \rangle+ \sum_{s \leq t} | \Delta M^i(t)|^2$. If $k \geq
  3$, we define $V^i_k(t)$ as the compensator of the $k$-variation
  process $\sum_{s \leq t} | \Delta M^i(t)|^k$ of $ M^i(t)$. Since
  $\Delta M^i(t) \in \{ 0, 1 \}$ for all $0 \leq t \leq 1$, the
  $V_k^i$ are all equal for $k \geq 3$ and such that $V^i_k(t) \leq
  V_2^i(t)$, for all $k \geq 3$. The process $S_h^i(t)$ has been
  defined as the compensator of~\eqref{eq:S_h_def}. As a consequence,
  we have:
   \begin{align*}
     S_h^i(t)
     &= % \frac{h^2}{2 m} \int_0^t \alpha(u, X_i)^2 dV^i_2(u) +
     \sum_{k \geq 2}\frac{1}{k!} \Big(\frac{h}{\sqrt{n}} \Big)^k
     \int_0^t |\alpha(u, X_i)|^k dV^i_k(u) \\
     &\leq \int_0^t \alpha(u, X_i)^2 dV^i_2(u) \times \sum_{k \geq
       2}\frac{\norm{\alpha}_\infty^{k-2}}{k!}
     \Big(\frac{h}{\sqrt{n}} \Big)^k
   \end{align*}
   and if $\langle Z_n(\alpha) \rangle \leq \delta^2$
   \begin{equation*}
     S_h(1) \leq \frac{n
       \delta^2}{\norm{\alpha}_\infty^2} \Big(
     \exp\Big( \frac{h \norm{\alpha}_\infty}{\sqrt{n}} \Big) - 1 -
     \frac{h \norm{\alpha}_\infty}{\sqrt{n}} \Big).
   \end{equation*}
   Thus, plugging this in~\eqref{eq:deviamart1} gives
   \begin{equation*}
     \psi_{n, \delta}(h) \leq \frac{n
       \delta^2}{\norm{\alpha}_\infty^2} \Big(
     \exp\Big( \frac{h \norm{\alpha}_\infty}{\sqrt{n}} \Big) - 1 -
     \frac{h \norm{\alpha}_\infty}{\sqrt{n}} \Big)
   \end{equation*}
   for any $h > 0$. Now, choosing
   \begin{equation*}
     h := \frac{\sqrt{n}}{\norm{\alpha}_\infty} \log
     \Big(\frac{z\norm{\alpha}_\infty }{\delta^2 \sqrt{n}} +1 \Big)
   \end{equation*}
   entails~\eqref{eq:cramer_minoration}.
 \end{proof}

 \subsection*{Proof of Theorem~\ref{thm:generic_chaining}}

 \begin{proof}[Proof of Theorem~\ref{thm:generic_chaining}]
   First, note that \eqref{eq:model} entails
   \begin{align*}
     \ell_{\alpha}(X, (Y_t), (N_t)) = \ell_\alpha'(X, (Y_t)) - 2
     \int_0^1 \alpha(t, X) d M(t),
   \end{align*}
   where $\ell_\alpha'$ is the loss function
   \begin{align*}
     \ell_\alpha'(x, (y_t)) := \int_0^1 \alpha(t, x)^2 y(t) dt - 2
     \int_0^1 \alpha(t, x) \alpha_0(t, x) y(t) dt.
   \end{align*}
   So, the following decomposition holds:
   \begin{equation*}
     (P - P_n)(\ell_\alpha - \ell_{\alpha_*}) = (P - P_n)(\ell_\alpha'
     - \ell_{\alpha_*}') + \frac{2}{\sqrt n} Z_n(\alpha_* - \alpha),
   \end{equation*}
   where we recall that $Z_n(\cdot)$ is given
   by~\eqref{eq:empirical_process_training}, and where
   $P(\ell_\alpha') := \E[ \ell_{\alpha}'(X, (Y_t)) ]$ and
   $P_n(\ell_\alpha') := \frac{1}{n} \sum_{i=1}^n \ell_{\alpha}'(X_i,
   (Y_t^i))$. First, let us prove the concentration inequality for
   $\sup_{\alpha \in A} ( Z_n(\alpha_*) - Z_n(\alpha))$. The proof
   follows the lines of the proof of Theorem~1.2.7 in
   \cite{talagrand06}. Consider admissible sequences $(\mathcal
   B_j)_{j \geq 0}$ and $(\mathcal C_j)_{j \geq 0}$ such that
  \begin{align*}
    \sum_{j \geq 0} 2^{j} \Delta(B_j(\alpha), d_\infty) \leq 2
    \gamma_1(A, d_\infty) \; \text{ and } \; \sum_{j \geq 0} 2^{j/2}
    \Delta(C_j(\alpha), d_2) \leq 2 \gamma_2(A, d_2)
  \end{align*}
  for any $\alpha \in A$. We construct partitions $\mathcal A_j$ of
  $A$ as follows. Set $\mathcal A_0 = \{ A \}$ and for $j \geq 1$,
  $\mathcal A_j$ is the partition generated by $\mathcal B_{j-1}$ and
  $\mathcal C_{j-1}$, namely the partition consisting of every set $B
  \cap C$ where $B \in \mathcal B_{j-1}$ and $C \in \mathcal
  C_{j-1}$. Note that $|\mathcal A_j| \leq (2^{2^{j-1}})^2 = 2^{2^j}$
  so that $(\mathcal A_j)$ is admissible. Define a sequence $(A_j)_{j
    \geq 0}$ of increasing subsets of $A$ by taking exactly one
  element in each set of $\mathcal A_j$. Such a set $A_j$ is then used
  as an approximation of $A$, and is such that $|A_j| \leq
  2^{2^j}$. Define $\pi_j(\alpha)$ by the relation
  \begin{equation*}
    A_j \cap A_j(\alpha) = \{ \pi_j(\alpha) \},
  \end{equation*}
  and take $\pi_0(\alpha) = \alpha_*$. In view of
  Lemma~\ref{lem:tricky_deviation}, we have with a probability larger
  than $1 - 2 \exp(-(x + 2^{j+1}))$:
  \begin{align*}
    Z_n( \pi_{j-1}(\alpha) ) - Z_n( \pi_j(\alpha)) &\leq C_0
    d_2(\pi_j(a), \pi_{j-1}(\alpha)) \sqrt{x + 2^{j+1}} \\
    &+ (C_0 + 1) \frac{ d_\infty(\pi_j(\alpha), \pi_{j-1}(\alpha)) (x
      + 2^{j+1})}{\sqrt n}.
  \end{align*}
  Now, for a fixed $\alpha \in A$, decompose the increment
  $Z_n(\alpha_*) - Z_n(\alpha)$ along the \emph{chain}
  $(\pi_j(\alpha))_{j \geq 0}$:
  \begin{equation*}
    Z_n(\alpha_*) - Z_n(\alpha) = \sum_{j \geq 1} \big(
    Z_n(\pi_{j-1}(\alpha)) - Z_n(\pi_{j}(\alpha)) \big),
  \end{equation*}
  and note that the number of pairs $\{ \pi_j(\alpha),
  \pi_{j-1}(\alpha) \}$ is at most $2^{2^j} \times 2^{2^{j-1}} \leq
  2^{2^{j+1}}$. This gives, together with union bounds for each term
  of the chain:
  \begin{align*}
    \sup_{\alpha \in A} (Z_n(\alpha_*) - Z_n(\alpha)) \leq
    \sup_{\alpha \in A} \sum_{j \geq 1} \Big( &C_0 \sqrt{x + 2^{j+1}}
    d_2(\pi_j(\alpha),
    \pi_{j-1}(\alpha)) \\
    &+ \frac{C_0 + 1}{\sqrt n} (x + 2^{j+1}) d_\infty(\pi_j(\alpha),
    \pi_{j-1}(\alpha)) \Big)
  \end{align*}
  with a probability larger than $1 - 2 \sum_{j \geq 1} 2^{2^{j+1}}
  \exp(-(x + 2^{j+1})) \geq 1 - L \exp(-x)$ (with $L \approx 0.773$).
  % , where $L := 2 \sum_{j \geq 2} (2 / e)^{2^j}$.
  But, for any $j \geq 2$, $\pi_j(\alpha), \pi_{j-1}(\alpha) \in
  A_{j-1}(\alpha) \subset B_{j-2}(\alpha)$, so
  $d_\infty(\pi_j(\alpha), \pi_{j-1}(\alpha)) \leq
  \Delta(B_{j-2}(\alpha), d_\infty)$ and $d_\infty(\pi_1(\alpha),
  \pi_0(\alpha)) \leq \Delta(B_0(\alpha), d_\infty) = \Delta(A,
  d_\infty)$. Doing the same for $d_2$, we obtain that, with
  probability $\geq 1 - L \exp(-x)$:
  \begin{equation}
    \label{eq:deviation_sup_Z_n}
    \sup_{\alpha \in A} (Z_n(\alpha_*) - Z_n(\alpha)) \leq 2 C_0 (1 +
    \sqrt{x}) \gamma_2(A, d_2) + \frac{2(C_0 + 1)}{\sqrt n} (1 + x)
    \gamma_1(A, d_\infty).
  \end{equation}
  We can do the same job for $\sup_{\alpha \in A} (P -
  P_n)(\ell_\alpha' - \ell_{\alpha_*}')$. Note that
  \begin{align*}
    &\ell_\alpha(X, (Y_t)) - \ell_{\alpha_*}(X, (Y_t)) \\
    &= \int_0^1 ( \alpha(t, X) - \alpha_*(t, X)) (\alpha(t, X) +
    \alpha_*(t, X) - 2 \alpha_0(t, X) ) Y(t) dt,
  \end{align*}
  so using Assumptions~\ref{ass:support} and~\ref{ass:countable}, we
  have $| \ell_\alpha(X, (Y_t)) - \ell_{\alpha_*}(X, (Y_t)) | \leq 2(b
  + \norm{\alpha_0}_\infty) \norm{\alpha - \alpha_*}_\infty$ and
  \begin{align*}
    \E [\ell_\alpha(X, (Y_t)) &- \ell_{\alpha_*}(X, (Y_t)) )^2 ] \leq
    4(b + \norm{\alpha_0}_\infty)^2 \norm{\alpha - \alpha_*}^2.
  \end{align*}
  Therefore, the Bernstein's inequality (for the sum of i.i.d. random
  variables) entails that
  \begin{equation*}
    (P - P_n)(\ell_\alpha' - \ell_{\alpha_*}') \leq 2 (b +
    \norm{\alpha_0}_\infty ) \Big( \frac{\norm{\alpha
        - \alpha_*} \sqrt{2x}}{\sqrt n} + \frac{\norm{\alpha -
        \alpha_*}_\infty x} {n} \Big)
  \end{equation*}
  holds with a probability larger than $1 - e^{-x}$. Then, we can
  apply again the generic chaining argument to prove that with a
  probability larger than $1 - L e^{-x}$:
  \begin{equation*}
    \sup_{\alpha \in A} (P - P_n)(\ell_\alpha' - \ell_{\alpha_*}')
    \leq  4 (b + \norm{\alpha_0}_\infty) \Big( \frac{\gamma_2(A, d_2)
      (1 + \sqrt{x})}{\sqrt n}  + \frac{\gamma_1(A, d_\infty) (1 +
      x)}{n} \Big).
  \end{equation*}
  This concludes the proof of the Theorem.
  % Thus, we have for any $y > 0$:
%   \begin{align*}
%     \P^n \Big[ \sup_{a \in A} Z_n(a) - Z_n(a_0) &\leq 2 C_0
%     \gamma_2(A, d_2) + \frac{2(C_0 + 1)}{\sqrt n} \gamma_1(A,
%     d_\infty) + y \Big] \\
%     &\geq 1 - L \exp \Big( -\frac{y^2}{16 C_0^2 \gamma_2(A, d_2)^2}
%     \wedge \frac{y \sqrt n}{4 (C_0 + 1) \gamma_1(A, d_\infty)} \Big),
%   \end{align*}
%   and the conclusion follows by integration of this inequality.
\end{proof}

\subsection*{Proof of Theorem~\ref{thm:oracle}}

\begin{proof}[Proof of Theorem~\ref{thm:oracle}]
%   First, let us recall some notations and let us gather some of the
%   properties we obtained earlier. We have a finite dictionary of
%   functions $A(\Lambda) = \{ \alpha_\lambda \;;\; \lambda \in \Lambda
%   \}$ with cardinality $M$ such that $\norm{\alpha_\lambda -
%     \alpha_0}_\infty \leq Q$ for all $\lambda \in \Lambda$, and we
%   assumed that $\norm{\alpha_0}_\infty \leq Q$, where $Q > 0$ is a
%   fixed constant. % We recall that $n$ is
%   % the learning sample size and that $m$ is the training sample size
%   % ($n + m = n$).
%   The empirical risk $R_{n}(\alpha)$ is given
%   by~\eqref{eq:empirial_risk}, and we know from~\eqref{eq:norm_risk}
%   that the risk $R(\alpha) = E [ R_{n}(\alpha) ]$ satisfies
%   \begin{equation*}
%     R(\alpha) - R(\alpha_0) = \norm{\alpha - \alpha_0}^2,
%   \end{equation*}
%   where we recall that the norm $\norm{\cdot}$ is given
%   by~\eqref{eq:norm}. The quantity $R(\alpha) - R(\alpha_0)$ is
%   called \emph{excess risk}.
  Recall that the \emph{linearized risk} over $A(\Lambda)$ is given by
  \begin{equation*}
    \mathsf R(\theta) := \sum_{\lambda \in \Lambda} \theta_\lambda
    P(\ell_{\alpha_\lambda})
  \end{equation*}
  for $\theta \in \Theta$, where we recall that
  \begin{equation*}
    \Theta = \{ \theta \in \mathbb R^{M} : \theta_\lambda
    \geq 0,\; \sum_{\lambda \in \Lambda} \theta_\lambda = 1 \},
  \end{equation*}  
  and the \emph{linearized empirical risk} is given by
  \begin{equation*}
    \mathsf R_{n}(\theta) = \sum_{\lambda \in \Lambda}
    \theta_\lambda P_{n}(\ell_{\alpha_\lambda}).
  \end{equation*}
  We recall that the mixing estimator $\hat \alpha$ is given by
  \begin{equation*}
    \hat \alpha := \sum_{\lambda \in \Lambda} \hat \theta_\lambda
    \alpha_\lambda,
  \end{equation*}
  where the Gibbs weights $\hat \theta = (\hat
  \theta_\lambda)_{\lambda \in \Lambda} :=
  (\theta(\alpha_\lambda))_{\lambda \in \Lambda}$ are given
  by~\eqref{eq:weights} and are the unique solution of the
  minimization problem~\eqref{eq:pena_linearized_risk}.
%   \begin{equation*}
%     \min_{\theta \in \Theta} \Big\{ \mathsf R_{n}(\theta) +
%     \frac{T}{n} \sum_{\lambda \in \Lambda} \theta_\lambda \log
%     \theta_\lambda \Big\},
%   \end{equation*}
%   where we use the convention $0 \log 0 = 0$.
  By convexity of the risk, we have for any $\epsilon > 0$:
  \begin{equation*}
    P(\ell_{\hat \alpha} - \ell_{\alpha_0}) \leq (1 + \epsilon)
    (\mathsf R_{n}(\hat \theta) - P_{n}(\ell_{\alpha_0})) + \mathcal
    R_{n},
  \end{equation*}
  where we introduced the residual term
  \begin{align*}
    \mathcal R_{n} &:= \mathsf R(\hat \theta) - P(\ell_{\alpha_0}) -
    (1 + \epsilon) (\mathsf R_{n}(\hat \theta) -
    P_{n}(\ell_{\alpha_0})) \\
    &= \sum_{\lambda \in \Lambda} \hat \theta_\lambda \Big( P
    (\ell_{\alpha_\lambda} - \ell_{\alpha_0} ) - (1 + \epsilon)
    P_{n}(\ell_{\alpha_\lambda} - \ell_{\alpha_0}) \Big).
  \end{align*}
  Let $\hat \lambda$ be such that $\alpha_{\hat \lambda}$ is the
  empirical risk minimizer in $A(\Lambda)$, namely
  \begin{equation*}
    P_{n}(\ell_{\alpha_{\hat \lambda}}) = \min_{\lambda \in \Lambda}
    P_{n}(\ell_{\alpha_\lambda}).
  \end{equation*}
  Since
  \begin{equation*}
    \sum_{\lambda \in \Lambda} \hat \theta_\lambda \log \Big( \frac{\hat
      \theta_\lambda}{1 / M} \Big) = K(\hat \theta, u) \geq 0,
  \end{equation*}
  where $K(\hat \theta, u)$ denotes the Kullback-Leibler divergence
  between the weights $\hat \theta$ and the uniform weights $u := (1 /
  M)_{\lambda \in \Lambda}$, we have
  \begin{align*}
    \mathsf R_{n}(\hat \theta) &\leq \mathsf R_{n}(\hat \theta) +
    \frac{T}{n} K(\hat \theta, u) \\
    &= \mathsf R_{n}(\hat \theta) + \frac Tn \sum_{\lambda \in
      \Lambda} \hat \theta_\lambda \log \hat \theta_\lambda + \frac{T
      \log M}{n} \\
    &\leq \mathsf R_{n}(e_{\hat \lambda}) + \frac{T \log M}{n} \\
    &= P_{n}(\ell_{\alpha_{\hat \lambda}}) + \frac{T \log M}{n},
  \end{align*}
  where $e_\lambda \in \Theta$ is the vector with all its coordinates
  equal to $0$ excepted for the $\lambda$-th which is equal to
  $1$. This gives
  \begin{equation*}
    P(\ell_{\hat \alpha} - \ell_{\alpha_0}) \leq (1 + \epsilon)
    \min_{\lambda \in \Lambda} P_{n}(\ell_{\alpha_\lambda} -
    \ell_{\alpha_0} ) + \mathcal R_{n},
  \end{equation*}
  and consequently
  \begin{align*}
    \E^{n} \norm{\hat \alpha - \alpha_0}^2 &\leq (1 + \epsilon)
    \min_{\lambda \in \Lambda} \norm{\alpha_\lambda - \alpha_0}^2 + (1
    + \epsilon) \frac{T \log M}{n} + \E^{n}[ \mathcal R_{n} ].
  \end{align*}
  Hence, it remains to prove that for some constant $C = C_{\epsilon,
    b, \norm{\alpha}_\infty}$, we have
  \begin{equation}
    \label{eq:residual_term}
    \E^{n} [ \mathcal R_{n} ] \leq \frac{C \log M}{n}.
  \end{equation}
  Since $\mathsf R(\cdot)$ and $\mathsf R_{n}(\cdot)$ are linear on
  $\Theta$, we have
  \begin{equation*}
    \mathcal R_{n} \leq \max_{\alpha \in A(\Lambda)} \Big( (1 +
    \epsilon) \big(P(\ell_{\alpha} - \ell_{\alpha_0}) -
    P_{n}(\ell_{\alpha} - \ell_{\alpha_0}) \big) - \epsilon
    P(\ell_{\alpha} - \ell_{\alpha_0}) \Big).
  \end{equation*}
  The following decomposition holds (see
  Section~\ref{sec:erm_risk_bound}):
  \begin{equation*}
    (P - P_n)(\ell_\alpha - \ell_{\alpha_0}) = (P - P_n)(\ell_\alpha'
    - \ell_{\alpha_0}') + \frac{2}{\sqrt n} Z_n(\alpha_0 - \alpha).
  \end{equation*}
  The Bernstein's inequality for the sum of i.i.d. variables (see the
  proof of Theorem~\ref{thm:generic_chaining}) gives
  \begin{equation*}
    (P - P_n)(\ell_\alpha' - \ell_{\alpha_0}') \leq (b +
    \norm{\alpha_0}_\infty ) \Big( \frac{\norm{\alpha
        - \alpha_0} \sqrt{2x}}{\sqrt n} + \frac{\norm{\alpha -
        \alpha_0}_\infty x} {n} \Big),
  \end{equation*}
  so together with Lemma~\ref{lem:tricky_deviation}, and since
  $P(\ell_\alpha - \ell_{\alpha_0}) = \norm{\alpha - \alpha_0}^2$, we
  obtain that
  \begin{align*}
    (P - P_n)(\ell_\alpha - \ell_{\alpha_0}) \leq
    \frac{C_{\norm{\alpha_0}_\infty, b}^1 \sqrt{2 x P(\ell_\alpha -
        \ell_{\alpha_0})}}{\sqrt n} + \frac{C_{\norm{\alpha_0}_\infty,
        b}^2 x} {n}
  \end{align*}
  with probability larger than $1 - 3 e^{-x}$, where
  $C_{\norm{\alpha_0}_\infty, b}^1 := C_{\norm{\alpha_0}_\infty} /
  \sqrt 2 + b + \norm{\alpha_0}_\infty$ and
  $C_{\norm{\alpha_0}_\infty, b}^2 := (C_{\norm{\alpha_0}_\infty} + 1
  + b + \norm{\alpha}_\infty) (b + \norm{\alpha_0}_\infty)$, with
  $C_{\norm{\alpha_0}_\infty}$ given in
  Lemma~\ref{lem:tricky_deviation}. Now, using the fact that
  \begin{equation*}
    \frac{C_{\norm{\alpha_0}_\infty, b}^1 \sqrt{2 x P(\ell_\alpha -
        \ell_{\alpha_0})}}{\sqrt n} \leq \frac{\epsilon}{1 + \epsilon}
    P(\ell_\alpha - \ell_{\alpha_0}) +   \frac{(1 + \epsilon)
      (C_{\norm{\alpha_0}_\infty, b}^1)^2}{\epsilon} \frac{x}{n},
  \end{equation*}
  we obtain that with a probability larger than $1 - 3e^{-x}$:
  \begin{equation*}
    (1 +
    \epsilon) \big(P(\ell_{\alpha} - \ell_{\alpha_0}) -
    P_{n}(\ell_{\alpha} - \ell_{\alpha_0}) \big) - \epsilon
    P(\ell_{\alpha} - \ell_{\alpha_0}) \leq C_{\epsilon,
      \norm{\alpha_0}_\infty, b} \frac x n,
  \end{equation*}
  where $C_{\epsilon, \norm{\alpha_0}_\infty, b} :=
  (C_{\norm{\alpha_0}_\infty, b}^1)^2 (1 + \epsilon)^2 / \epsilon + (1
  + \epsilon) C_{\norm{\alpha_0}_\infty, b}^2$. This subexponential
  deviation entails that for any $x > 0$:
  \begin{equation*}
    \E^{n} \big[ \mathcal R_{n} \big] \leq 2 x + \frac{3 M C \exp(-n
      x / C)}{n},
  \end{equation*}
  where $C = C_{\epsilon, \norm{\alpha_0}_\infty, b}$. If we denote by
  $x(y)$ the unique solution of $x = y \exp(-x)$, where $y > 0$, we
  obtain
  \begin{equation*}
    \E^{n} \big[ \mathcal R_{n} \big] \leq \frac{5 C \log M}{n}
  \end{equation*}
  for the choice $x = C x(M) / n$, since we have $x(M) \leq \log
  M$. This concludes the proof of Theorem~\ref{thm:oracle}.
\end{proof}

\begin{proof}[Proof of Theorem~\ref{thm:adaptivesim}]
  Assume for now that~\eqref{eq:SImodel} holds. Take $v_\Delta \in
  S_\Delta^+$ such that $|v_\Delta - v_0|_2 \leq \Delta$, and let $m^*
  = (m_1^*, m_2^*)$ be the oracle dimension of the sieve for the link
  function, that satisfies~\eqref{eq:oracle_dimension} with $d =
  1$. Denote for short the oracle estimator
  \begin{equation*}
    \bar \alpha_* = \bar \beta_{m^*, v_\Delta}(\cdot, v_\Delta^\top
    \cdot),
  \end{equation*}
  that is, the element of $A_{m^*}$ that minimizes the empirical risk
  computed using the training sample $D_{n, 1}(v_\Delta)$.

  Note that the cardinality of $S_\Delta^+$ is smaller than $c /
  \Delta^{d-1}$, where $\Delta = (n \log n)^{-1/2}$, so the
  cardinality of the whole dictionary $\{ \bar \alpha_m : m \in
  \mathcal M_n \} \cup \{ \bar \alpha_{m, v}^{\rm SIM} : m \in
  \mathcal M_n^{\rm SIM}, v \in S_\Delta^{d-1} \}$ is of order $c
  n^{(d-1) / 2} (\log n)^{2 + 1/2} + (\log n)^{d+1}$. As a
  consequence, Theorem~\ref{thm:oracle} gives
  \begin{equation*}
    \E^{2n} \norm{\hat \alpha_n - \alpha_0}^2 \leq 2 \E^n \norm{\bar
      \alpha_* - \alpha_0}^2 + c \frac{\log n}{n}.
  \end{equation*}
  Note that~\eqref{eq:link_lip} entails $\norm{\beta_0(\cdot,
    v_\Delta^\top \cdot) - \beta_0(\cdot, v_0^\top \cdot)}^2 \leq c
  \Delta^2 = c / (n \log n)$. Hence,
  \begin{equation*}
    \norm{\bar \alpha_* - \alpha_0}^2 \leq 2 \norm{\bar \alpha_* -
      \beta_0(\cdot, v_\Delta^\top \cdot)}^2 + \frac{2 c}{n \log n}.
  \end{equation*}
  We shall denote in what follows by $\E_v^n$ the expectation wrt
  $\P_v^n$, the joint law of the observations when the intensity
  writes $\beta_0(\cdot, v^\top \cdot)$ (the true index is $v$). For
  two indexes $v, v_0 \in S_+^{d-1}$, we introduce the following
  likelihood ratio:
  \begin{equation*}
    L_n(v_0, v) = \frac{ \d \P_{\beta_0(\cdot, v_0^\top
        \cdot)}^n}{\d \P_{\beta_0(\cdot, v^\top \cdot)}^n} ,
  \end{equation*}
  which is the likelihood ratio of the training data $D_{n, 1}$
  ``between'' the two indexes $v$ and $v_0$. It can be explicitly
  computed using Jacod's formula, see
  Appendix~\ref{sec:appendix_proba} below. Of course, when $v$ and
  $v_0$ are close to each other, we expect $L_n(v_0, v)$ to be
  small. This is the statement of the next Lemma.
  \begin{lemma}
    \label{lem:devialikeli}
    Grant Assumption~\ref{ass:sim}, and let $v, v_0 \in S_+^{d-1}$ be
    such that $\norm{v - v_0}_2 \leq \Delta_n$, where $\Delta_n = (n
    \log n)^{-1/2}$. Then, if $n$ is large enough, one has for any $x >
    0$:
    \begin{equation*}
      \P_{v_0}^n[ L_n(v_0, v) \geq x ] \leq \sqrt x n^{-c (\log x)^2},
    \end{equation*}
    where $c = b_0 / (2 d c_0^2)$.
  \end{lemma}
  The proof of this Lemma can be found below. It uses the same kind of
  arguments as the proof of
  Proposition~\ref{prop:deviation_martingale}. Let $x > 0$ to be
  chosen later on, and decompose the expectation over $\{ L_n(v_0,
  v_\Delta) > x \}$ and $\{ L_n(v_0, v_\Delta) \leq x \}$ to get
  \begin{align*}
    \E_{v_0}^{n} \norm{\bar \alpha_* - \beta_0(\cdot, v_\Delta^\top
      \cdot)}^2 &= \E_{v_\Delta}^{n} [ \norm{\bar \alpha_* -
      \beta_0(\cdot, v_\Delta^\top \cdot)}^2 \ind{L_n(v_0, v_\Delta)
      \leq x} L_n(v_0, v_\Delta)] \\
    &+ \E_{v_0}^{n} [ \norm{\bar \alpha_* - \beta_0(\cdot,
      v_\Delta^\top \cdot)}^2 \ind{L_n(v_0, v_\Delta) > x} ]
  \end{align*}
  so using Assumption~\ref{def:collection} and
  Lemma~\ref{lem:devialikeli}, we obtain
  \begin{align*}
    \E_{v_0}^{n} \norm{\bar \alpha_* - \beta_0(\cdot, v_\Delta^\top
      \cdot)}^2 \leq x \E_{v_\Delta}^{n} \norm{\bar \alpha_* -
      \beta_0(\cdot, v_\Delta^\top \cdot)}^2 + 4 b^2 \sqrt{x} n^{-c
      (\log x)^2},
  \end{align*}
  so for $x = e^{1 / \sqrt c}$, we have
  \begin{align*}
    \E_{v_0}^{n} \norm{\bar \alpha_* - \beta_0(\cdot, v_\Delta^\top
      \cdot)}^2 \leq c \E_{v_\Delta}^{n} \norm{\bar \alpha_* -
      \beta_0(\cdot, v_\Delta^\top \cdot)}^2 + \frac c n.
  \end{align*}
  But, $\E_{v_\Delta}^{n} \norm{\bar \alpha_* - \beta_0(\cdot,
    v_\Delta^\top \cdot)}^2$ is nothing but the risk of the minimizer
  $\bar \beta_{m^*}$ of the empirical risk $R_{n, 1}^{(v_\Delta)}$
  over the sieve $A_{m^*}$: in this risk, the ``true covariate'' is
  now $v_\Delta^\top X$. Indeed,
  \begin{align*}
    \E_{v_\Delta}^{n} &\norm{\bar \alpha_* - \beta_0(\cdot,
      v_\Delta^\top \cdot)}^2 \\
    &= \E_{v_\Delta}^{n} \Big[ \int_0^1 \int (\bar \beta_{m^*}(t,
    v_\Delta^\top x) - \beta_0(t, v_\Delta^\top x) )^2 \E[ Y(t) | X =
    x ] dt P_X(dx) \Big] \\
    &= \E_{v_\Delta}^{n} \Big[ \int_0^1 \int (\bar \beta_{m^*}(t, x')
    - \beta_0(t, x') )^2 \E[ Y(t) | v_\Delta^\top X = x' ] dt
    P_{v_\Delta^\top X}(dx') \Big],
  \end{align*}
  so conducting the same analysis as in
  Section~\ref{sec:purely_nonparametric}, we can prove that the choice
  of $m^*$ entails that
  \begin{equation*}
    \E_{v_\Delta}^{n} \norm{\bar \alpha_* - \beta_0(\cdot,
      v_\Delta^\top \cdot)}^2 \leq c n^{-2 \bar {\bs s} /(2 \bar {\bs
        s} + 2)}.
  \end{equation*}
  This concludes the proof of Theorem~\ref{thm:adaptivesim} in the
  single-index case. If~\eqref{eq:SImodel} does not hold, then in the
  oracle inequality we take the oracle purely nonparametric element,
  using the same analysis as in
  Section~\ref{sec:purely_nonparametric}.
\end{proof}

% \section{Proof of the Lemmas}
% \label{lem:lemmas_proof}

\begin{proof}[Proof of Lemma~\ref{lem:devialikeli}]
  In view of Equation~\eqref{eq:log_likeli}, see
  Appendix~\ref{sec:appendix_proba}, we can write,
  using~\eqref{eq:model}:
  \begin{align*}
    \log L_n(v_0, v) &= \sum_{i=1}^n \int_0^1 \Big( \mathcal L_{v_0,
      v}(t, X_i) d N^i(t) - \Upsilon_{v_0,v}(t, X_i) Y^i(t) dt
    \Big) \\
    &= \sum_{i=1}^n \int_0^1 \mathcal L_{v_0, v}(t, X_i) d M^i(t)  \\
    &+ \sum_{i=1}^n \int_0^1\Big\{\mathcal L_{v_0,v}(t, X_i)
    \beta_0(t, v_0^\top X_i) - \Upsilon_{v_0,v}(t, X_i) \Big\}Y^i(t)
    dt ,
  \end{align*}
  where we shall use the notations
  \begin{align*}
    \Upsilon_{v_0,v}(t, X_i) &:= \beta_0(t, v_0^\top X_i) - \beta_0(t,
    v^\top X_i) \\
    \mathcal L_{v_0, v}(t, X_i) &:= \log \beta_0(t, v_0^\top X_i) -
    \log \beta_0(t, v^\top X_i)
  \end{align*}
  throughout the proof the Lemma. Now, fix some $h > 0$ (to be chosen
  later on) and write
  \begin{align*}
    &\P_{v_0}^n [ L_n(v_0, v) \geq x ] \\
    &\leq \E_{v_0}^n [ L_n(v_0, v)^h ] e^{-h \log x} \\
    &= \E_{v_0}^n \Big[ \exp \Big( \sum_{i=1}^n h \int_0^1 \mathcal
    L_{v_0, v}(t, X_i) d M^i(t) \\
    &+ h \sum_{i=1}^n \int_0^1 \Big\{ \mathcal L_{v_0,v}(t, X_i)
    \beta_0(t, v_0^\top X_i) - \Upsilon_{v_0,v}(t, X_i) \Big\} Y^i(t)
    dt - h \log x \Big) \Big].
  \end{align*}  
  We follow the main steps of the proof of
  Proposition~\ref{prop:deviation_martingale}. Define
  \begin{align*}
    \tilde U_h^i(t) := h \int_0^t \mathcal L_{v_0, v}(s, X_i) d M^i(s)
    - \tilde S_h^i(t):=h O^i(t)- \tilde S_h^i(t),
  \end{align*} 
  where $\tilde S_h^i(t)$ is the compensator of
  \begin{equation}
    \label{eq:tildeS_h_def}
    \frac12 h^2 \langle O^{i,c}( t) \rangle + \sum_{s \leq t}
    \left(\exp(h | \Delta O^i(s)|) -1 -h| \Delta O^i(s)| \right),
  \end{equation} 
  where $O^{i,c}$ is the continuous part of the process $O^i.$ We know
  from the proof of Lemma~2.2 and Corollary~2.3 of~\cite{vandegeer95},
  see also~\cite{lipstershiryayev}, that $\exp(\tilde U_h^i) = \exp(h
  O^i- \tilde S_h^i)$, for $i =1, \dots, n$ are
  i.i.d. super-martingales. As a consequence, we get:
  \begin{equation*}
    \P_{v_0} [ L_n(v_0, v) \geq x ] \leq \E_{v_0}^{1/2} \Big[ \exp \Big( 2
    \sum_{i=1}^n  \tilde U_h^i(t) \Big) \Big] \, \E_{v_0}^{1/2} [
    \mathbb L_n(1) ] \leq \E_{v_0}^{1/2} [ \mathbb L_n(1) ],
  \end{equation*}
  where
  \begin{align*}
    \mathbb L_n(1) := \exp \Big( &2 \sum_{i=1}^n \Big\{ \tilde
    S_h^i(1) + h \int_0^1 \Big\{\mathcal L_{v_0,v}(t, X_i) \beta_0(t,
    v_0^\top X_i) - \Upsilon_{v_0,v}(t, X_i) \Big\} Y^i(t) dt \Big\}
    \\
    &- 2 h \log C\Big).
  \end{align*}
  We are now establishing an upper bound for $\E_{v_0}^{1/2} [ \mathbb
  L_n(1) ]$. Looking closer to the process $ \tilde S_h^i(t)$, we can
  write:
  \begin{align*}
    \tilde S_h^i(t) = \sum_{k \geq 2} \frac{h^k}{k!}  \int_0^t |
    \mathcal L_{v_0, v}(s, X_i) |^k dV_k^i(s),
  \end{align*} 
  where the processes $V_k^i$ have been defined in the proof of
  Proposition~\ref{prop:deviation_martingale}. Assumption~\ref{ass:sim}
  and the fact that $\norm{v - v_0}_2 \leq \Delta$ gives
  \begin{equation}
    \label{eq:close_betas}
    | \Upsilon_{v_0,v}(t, x) | \leq c_0 \sqrt d \Delta := \epsilon
  \end{equation}
  for any $t \geq 0$ and $x \in [0, 1]^d$. In particular, we have $|
  \Upsilon_{v_0,v}(t, x) | \leq b_0 / 2$ when $n$ is large
  enough. This allows to write:
  \begin{align*}
    | \mathcal L_{v_0, v}(t, X_i) |\leq \Psi_{1 / \beta_0(t, v_0\T
      X_i)}( \Upsilon_{v_0,v}(t, X_i) )\times (1/\beta_0(t, v_0\T
    X_i))
  \end{align*}
  where $\Psi_a(x) := -\log(1 - a x) / a$ for $a > 0$ and $x <
  1/a$. Since $\Psi_a(x)= -\log(1 - a x) / a \leq x + a x^2$ for any
  $x \in [0, 1 / (2a)]$, we obtain
  \begin{align*}
    | \mathcal L_{v_0, v}(t, X_i) | &\leq \frac{|\Upsilon_{v_0,v}(t,
      X_i)|}{\beta_0(t, v_0\T X_i) \wedge \beta_0(t, v\T X_i)} \Big(1+
    \frac{|\Upsilon_{v_0,v}(t, X_i)|}{\beta_0(t, v_0\T
      X_i)\wedge\beta_0(t, v\T X_i)} \Big) \\
    &\leq \Big( \frac{\epsilon}{b_0} \Big) \Big(1 +
    \frac{\epsilon}{b_0}\Big).
  \end{align*}
  We can write, as a consequence:
  \begin{align*}
    \tilde S_h^i(t) &= \sum_{k \geq 2} \frac{h^k}{k!}  \int_0^t |
    \mathcal L_{v_0, v}(s, X_i) |^k dV_k^i(s) \\
    &\leq \int_0^t | \mathcal L_{v_0, v}(s, X_i) |^2 \beta_0(s, v_0\T
    X_i) Y^i(s) ds \times \sum_{k \geq 2} \frac{h^k}{k!}
    \Big(\frac{\epsilon}{b_0}\Big)^{k-2}
    \Big(1 + \frac{\epsilon}{b_0}\Big)^{k-2} \\
    &\leq \int_0^t | \mathcal L_{v_0, v}(s, X_i) |^2 \beta_0(s, v_0\T
    X_i) Y^i(s) ds \times\frac{h^2}{2}(1 + c_h),
  \end{align*}
  where
  \begin{equation*}
    c_h := 2 \sum_{k \geq 1} \frac{h^k}{(k+2)!}
    \Big(\frac{\epsilon}{b_0} \Big)^{k}
    \Big(1 + \frac{\epsilon}{b_0}\Big)^{k}.
  \end{equation*}
  Note that $c_h \leq 1$ for $h \epsilon$ and $\epsilon$ small
  enough. We obtain:
  \begin{align*} 
    \E_{v_0}^n [ \mathbb L_n(1) ] \leq \E_{v_0} \Big[ \exp \Big( n
    &\Big\{ \int_0^1 | \mathcal L_{v_0, v}(t, X) |^2 \beta_0(t,
    v_0\T X) Y(t) dt \times h^2 (1 + c_h) \\
    &+ 2 h \int_0^1 \Big\{\mathcal L_{v_0, v}(t, X) \beta_0(t,
    v_0^\top X) - \Upsilon_{v_0,v}(t, X) \Big\} Y(t) dt \Big\}
    \\
    &- 2 h \log x \Big) \Big].
  \end{align*}
  Using again the above trick involving the function $\Psi_a$, we
  obtain:
  \begin{align*} 
    \mathcal L_{v_0,v}(t, X_i) \beta_0(t, v_0^\top X_i) -
    \Upsilon_{v_0,v}(t, X_i) \leq \frac{\Upsilon_{v_0,v}(t, X_i)
      ^2}{\beta_0(t, v\T X_i)}\leq \frac{\epsilon^2}{b_0}.
  \end{align*}
  Using the fact that $\log(x/y)^2 x \leq 2 \epsilon^2 / (x \wedge y)$
  for any $x, y > 0$ such that $|x - y| \leq \epsilon \leq (x \wedge
  y) / 2$ and $\epsilon > 0$ small enough [decompose over $\{ x \leq y
  \}$ and $\{ x > y \}$ and use again the previous majoration of
  $\Psi_{a}(x)$], we have in view of~\eqref{eq:close_betas}:
  \begin{equation*}
    \mathcal L_{v_0, v}(t, x)^2 \beta_0(t, v_0\T x) \leq \frac{2
      \epsilon^2}{b_0}
  \end{equation*}
  for any $t \geq 0$ and $x \in [0,1]^d$ and $\epsilon$ small enough.
  In fine, we get, using the fact that $Y^i\leq 1$,
  \begin{align*}
    \E_{v_0}^n [ \mathbb L_n(1) ] &\leq \E_{v_0} \Big[ \exp \Big( n
    \int_0^1 \Big\{ \frac{2 \epsilon^2 h^2}{b_0} + \frac{2 h
      \epsilon^2}{b_0} \Big\} Y^i(t) dt - 2 h \log x \Big) \Big] \\
    &\leq \exp \Big( \frac{2 n \epsilon^2 h}{b_0} (1 + h) - 2 h \log x
    \Big)
  \end{align*}
  for any $h> 0$, so for the choice $h = b_0 \log x / (2 n
  \epsilon^2)$, we obtain
  \begin{align*}
    \P_{v_0}^n [ L_n(v_0, v) \geq x ] &\leq \sqrt{x} \exp \Big(
    -\frac{b_0 (\log x)^2}{2 n \epsilon^2} \Big),
  \end{align*}
  and the conclusion follows, since $\Delta = 1 / \sqrt{n \log n}$ and
  $n \epsilon^2 = d c_0^2 / \log n$.
\end{proof}

\appendix
\normalsize

\section{Appendix}

\subsection{Some tools from approximation theory}
\label{sec:approximation}

Let us give two examples of sieves, that are spanned by localized
basis. In each case, we give the control on $\bar r(A)$ and we give a
standard but useful approximation result below. Note that other
examples of sieves are available, see \cite{bbm99} for instance.

\subsubsection{Piecewise polynomials}
\label{sec:dyadic_polynomials}

Fix $l_1, \ldots, l_{d+1} \in \mathbb N$ and $m_1, \ldots, m_{d+1} \in
\mathbb N$, and define the set $\mathscr R$ of rectangles
$\prod_{i=1}^{d+1} [ (j_i - 1) 2^{-m_i}, j_i 2^{-m_i}[$ for $0 \leq
j_i \leq 2^{m_i}$, $i = 0, \ldots, d+1$. So, $\mathscr R$ is a regular
partition of $[0, 1]^{d+1}$. Take $m = (m_1, \ldots, m_{d+1})$ and
define $A_m$ as the set of functions $f : [0, 1]^{d+1} \rightarrow
\mathbb R$ such that for any $R \in \mathscr R$, the restriction of
$f$ to $R$ coincides with a polynomial of degree not larger than $l_i$
in the $i$th coordinate, for $i = 1, \ldots, d+1$. The dimension of
$A_m$ is then
\begin{equation*}
  D_m := \prod_{i=1}^{d+1} 2^{m_i} (l_i + 1),
\end{equation*}
and using \cite{bbm99}, see Section~3.2.1, we have, since $\mathscr R$
is a regular partition,
\begin{equation*}
  \bar r(A_m) \leq c_{l_1, \ldots, l_{d+1}, d},
\end{equation*}
where $c_{l_1, \ldots, l_{d+1}, d} = (\prod_{i=1}^{d+1} (l_i + 1)(2
l_i + 1))^{1/2}$.

\subsubsection{Wavelets}
\label{sec:wavelets}

Consider a pair $\{ \phi, \psi \}$ of scaling function and wavelet,
where $\psi$ has $K$ vanishing moments. Then $\phi$ and $\psi$ have a
support width of at least $2 K - 1$, and there is a pair with minimal
support, see \cite{daubechies88}. This is the starting point of the
construction of an orthonormal wavelet basis of $\mathbb L^2[0, 1]$,
as proposed in \cite{cohen_daubechies_vial93}.  Roughly, the idea is
to retain the interior scaling functions (those ``far'' from the edges
$0$ and $1$), and to add adapted edge scaling functions, see~Section~4
and Theorem~4.4 in \cite{cohen_daubechies_vial93}. This construction
allows to keep the orthonormality of the system and the number of
vanishing moment unchanged, as well as the number $2^j$ of scaling
function at each resolution $j$. More precisely, if $l$ is such that
$2^l \geq 2 K$, consider for $j \geq l-1$:
\begin{align*}
 \Psi_{j, k} :=
 \begin{cases}
   \psi_{j, k}^0 &\text{ if } j \geq l \text{ and } k = 0, \ldots, K-1 \\
   \psi_{j, k} &\text{ if } j \geq l \text{ and } k = K, \ldots, 2^j
   - K - 1 \\
   \psi_{j, k}^1 &\text{ if } j \geq l \text{ and } k = 2^j - K,
   \ldots, 2^j - 1 \\
   \phi_{l, k}^0 &\text{ if } j = l-1 \text{ and } k = 0, \ldots, K-1 \\
   \phi_{l, k} &\text{ if } j = l-1 \text{ and } k = K, \ldots, 2^l
   - K - 1 \\
   \phi_{l, k}^1 &\text{ if } j = l-1 \text{ and } k = 2^l - K,
   \ldots, 2^l - 1 \\
 \end{cases}
\end{align*}
where $\phi_{j, k} = 2^{j/2} \phi(2^j \cdot - x)$ and $\psi_{j, k} =
2^{j/2} \psi(2^j \cdot - x)$ are the "interior" dilatations and
translations of $\{ \phi, \psi \}$, and $\phi_{j, k}^0, \psi_{j, k}^0,
\phi_{j, k}^1, \psi_{j, k}^1$ are, at each resolution $j$, dilatations
of $2 K$ edge scaling functions and wavelets ($K$ for each edge). We
know from \cite{cohen_daubechies_vial93} that the collection
\begin{equation*}
  W := \{ \Psi_{j,k} : j \geq l - 1, k = 0, \ldots, 2^j - 1 \}
\end{equation*}
is an orthonormal basis of $\mathbb L^2[0, 1]$, and the interior and
edge wavelets have $K$ vanishing moments. Let $W^{(i)}, i = 1, \ldots,
d+1$ be several collections $W$ based on pairs $\{ \phi^{(i)},
\psi^{(i)} \}$ (possibly with different numbers of vanishing
moments). Then, the collection
\begin{align*}
  \{ \otimes_{i=1}^{d+1} \Psi_{j_i, k_i}^{(i)} : j_i \geq l_i - 1, k_i
  = 0, \ldots, 2^{j_i} -1, i = 1, \ldots, d+1 \},
\end{align*}
where $ \otimes_{i=1}^{d+1} \Psi_{j_i, k_i}^{(i)}(x_1, \ldots,
x_{d+1}) = \prod_{i=1}^{d+1} \Psi_{j_i, k_i}^{(i)}(x_i)$, is an
orthonormal basis of $\mathbb L^2 [0, 1]^{d+1}$ that has suitable
approximation properties for a function with an anisotropic
smoothness, see below. Let $m = (m_1, \ldots, m_{d+1}) \in \mathbb
N^{d+1}$ be fixed, where $m_i \geq l_i$ for any $i \in \{ 1, \ldots,
d+1\}$, and define the sieve
\begin{equation}
  \label{eq:wavelet_sieve}
  A_m := \Span \{ \Psi_\lambda : \lambda \in \Lambda(m) \},
\end{equation}
where for $\lambda = (j_1, k_1, \ldots, j_{d+1}, k_{d+1})$,
\begin{equation*}
  \Psi_\lambda := \otimes_{i=1}^{d+1} \Psi_{j_i, k_i}^{(i)},
\end{equation*}
and where
\begin{align*}
  \Lambda(m) = \{ (j_1, k_1, \ldots, j_{d+1}, k_{d+1}) : l_i &- 1 \leq
  j_i \leq m_i, \\
  &k_i = 0, \ldots, 2^{j_i} -1, i = 1, \ldots, d+1 \}
\end{align*}
The dimension of $A_m$ is $\prod_{i=1}^{d+1} D_{m_i}$, where $D_{m_i}
= 2^{m_i} - 2^{l_i} \leq 2^{m_i}$. The control of $r(A_m)$ easily
follows from the fact that if the resolution levels $j_i \geq l_i$ are
fixed for any $i = 1, \ldots, d+1$, the tensor products $\Psi_\lambda$
have disjoint supports, excepted for a finite number of indexes $k_i$,
that depends only on the support of the scaling and mother wavelet
functions used in the construction of $W$. As a consequence, we have
\begin{equation*}
  r(A_m) \leq \frac{1}{\sqrt D_m} \sup_{\beta \neq 0}
  \frac{\norm{\sum_{\lambda \in \Lambda(m)} \beta_\lambda
      \psi_\lambda}_\infty}{|\beta|_\infty} \leq c_\Psi,
\end{equation*}
where $D_m = \prod_{i=1}^{d+1} D_{m_i}$, $|\beta|_\infty =
\sup_{\lambda \in \Lambda(m)} |\beta_\lambda|$ and where $c_\Psi$ is a
constant that depends only the scaling and mother wavelet functions
used in the construction of the basis, and not on the resolution level
$m$.

In the next section, we give the definition of the anisotropic Besov
space, and recall a useful approximation result. The definitions and
results presented here can be found in~\cite{triebel06}, in particular
in Chapter~5 which is about anisotropic spaces.

\subsubsection{Anisotropic Besov space, approximation}

Let $\{ e_1, \ldots, e_{d+1} \}$ be the canonical basis of $\mathbb
R^{d+1}$ and $\bs s = (s_1, \ldots, s_{d+1})$ with $s_i > 0$ be a
vector of directional smoothness, where $s_i$ corresponds to the
smoothness in direction $e_i$. If $k \in \mathbb N$ and $x \in \mathbb
R^{d+1}$, define
\begin{equation*}
  \mathscr D_{e}^{k} := \{ x \in \mathbb R^{d+1} : x + j e \in [0,
  1]^{d+1} \text{ for } j = 0 , \ldots, k \}.
\end{equation*}
If $f : [0, 1]^{d+1} \rightarrow \mathbb R$, we define $\Delta_e^k f$
as the difference of order $k \geq 1$ and step $e \in [0, 1]^{d+1}$,
given by $\Delta_e^1 f(x) = f(x + e) - f(x)$ and $\Delta_e^k f(x) =
\Delta_e^1(\Delta_e^{k-1}f)(x)$ for any $x \in \mathscr D_e^k$. We say
that $f \in \L^2[0, 1]^{d+1}$ belongs to the anisotropic Besov space
$B_{2, \infty}^{\bs s}([0, 1]^{d+1})$ if the semi-norm
\begin{equation}
  \label{eq:besov}
  |f|_{B_{2, \infty}^{\bs s}} := \sup_{t > 0} \Big(
  \sum_{i=1}^{d+1} t^{-s_i} \sup_{h : |h| \leq t} \Big(
  \int_{\mathscr D_{h e_i}^{k_i}} (\Delta_{h e_i}^{k_i} f(x))^2 dx
  \Big)^{1/2} \Big)
\end{equation}
is finite. We know that the norms
\begin{equation*}
  \norm{f}_{B_{2, \infty}^{\bs s}} := \norm{f}_2 + |f|_{B_{2, \infty}^{\bs s}}
\end{equation*}
are equivalent for any choice of $k_i > s_i$. Note that if $\bs s =
(s, \ldots, s)$ for some $s > 0$, then $B_{2, \infty}^{\bs s}$ is the
standard isotropic Besov space. Moreover, the embedding $B_{2, 2}^{\bs
  s} \subset B_{2, \infty}^{\bs s}$ holds. When $s = (s_1, \ldots,
s_{d+1})$ has integer coordinates, $B_{2, 2}^{\bs s}$ is the
anisotropic Sobolev space
\begin{equation*}
  B_{2, 2}^{\bs s} = W_2^{\bs s} = \Big\{ f \in \L^2 : \sum_{i=1}^{d+1}
  \Big\| \frac{\partial^{s_i} f}{\partial x_i^{s_i}} \Big\|_2 < \infty
  \Big\}.
\end{equation*}
If $\bs s$ has non-integer coordinates, then $B_{2, 2}^{\bs s}$ is the
anisotropic Bessel-potential space
\begin{equation*}
  H^{\bs s} = \Big\{ f \in \L^2 : \sum_{i=1}^{d+1} \Big\| (1 +
  |\xi_i|^2)^{s_i/2} \hat f(\xi) \Big\|_2 < \infty \Big\},
\end{equation*}
where $\hat f$ is the Fourier transform of $f$. If $f \in B_{2,
  \infty}^{\bs s}$, we can give a control on the approximation term
$\inf_{\alpha \in A} \norm{\alpha - \alpha_0}$, when $A$ is spanned by
piecewise polynomials or wavelets (see above). Indeed, the next Lemma
is a direct consequence of the Jackson's estimate given in
\cite{hochmuth02a}, together with definition~\eqref{eq:besov} of the
Besov space. Note that this Lemma can be also found in \cite{CGG} and
\cite{lacour07}.
\begin{lemma}
  \label{lem:approximation}
  Assume that $\alpha_0 \in B_{2,\infty}^{\bs s}$ where $\bs s = (s_1,
  \ldots, s_{d+1})$ and let $l_i \geq s_i$ for $i = 1, \ldots,
  d+1$. Let $A_m$ be either\textup:
  \begin{itemize}
  \item the piecewise polynomial sieve \textup(see
    Section~\ref{sec:dyadic_polynomials}\textup) with degrees $l_i$ in
    the $i$th coordinate\textup, based on a partition with rectangles
    of sidelengthes $2^{-m_i},$ or
  \item the wavelet sieve \textup(see
    Section~\ref{sec:wavelets}\textup)\textup, where the wavelets have
    $l_i$ vanishing moments in the $i$th coordinate.
  \end{itemize}
  Then, there is a constant $c = c_{\bs s, d} > 0$ such that
  \begin{eqnarray*}
    \inf_{\alpha \in A_m} \| \alpha - \alpha_0 \|_2 \leq c
    |\alpha_0|_{B_{2,\infty}^{\bs s}} \sum_{i=1}^{d+1} 2^{- s_i m_i}.
 \end{eqnarray*}
\end{lemma}

\subsection{Some tools from the theory of counting processes and
  stochastic calculus}
\label{sec:appendix_proba}

Let $P_{\alpha_0}$ be the joint law of $\{ (X, N(t), Y(t)) : t \in [0,
1] \}$ when~\eqref{eq:model} holds (the intensity is $\alpha_0$). We
want to explain why the log-likelihood ratio $\ell(\alpha, \alpha_0) :=
\log(d P_\alpha / d P_{\alpha_0})$ writes, when both $\alpha$ and
$\alpha_0$ are assumed to be positive on $[0, 1]^{d+1}$:
\begin{equation}
  \label{eq:log_likeli0}
  \ell(\alpha, \alpha_0) = \int_0^1 \log\Big( \frac{\alpha(t,
    X)}{\alpha_0(t, X)} \Big) d N(t) - \int_0^1 (\alpha(t, X) -
  \alpha_0(t, X)) Y(t) dt.
\end{equation}
This will entail that the log-likelihood ratio $\ell_n(\alpha,
\alpha_0) := \log(d P_\alpha^n / d P_{\alpha_0}^n)$ of the independent
sample~\eqref{eq:whole_sample} satisfies
\begin{equation}
  \label{eq:log_likeli}
  \ell_n(\alpha, \alpha_0) = \sum_{i=1}^n \Big( \int_0^1 \log\Big(
  \frac{\alpha(t, X_i)}{\alpha_0(t, X_i)} \Big) d N^i(t) - \int_0^1
  (\alpha(t, X_i) - \alpha_0(t, X_i)) Y^i(t) dt \Big).
\end{equation}
Equation~\eqref{eq:log_likeli} is useful in several parts of the paper
(dimension reduction and lower bounds).

First, we recall Jacod's formula (see~\cite{ABGK}) for the likelihood of
a counting process. It writes, for the likelihood of $N$:
\begin{align*}
  \prodi_{t \in [0, 1]} \Big\{ ( \alpha_0(t, X) Y(t) )^{\Delta
    N(t)} (1 - \alpha_0(t, X) Y(t) )^{1 - \Delta N(t)} \Big\} dt,
\end{align*}
where $\Delta N(t) = N(t) - N(t_-)$ and where $\prodi$ is the
product-integral, see~\cite{ABGK} for a definition. But $N$ has a
finite number of jumps on $[0, 1]$ and $\Delta N (t) \in \{ 0, 1\}$
for any $t \in [0, 1]$, thus $1 - \Delta N(t) = 1$ for any $t \in [0,
1]$ excepted a finite number of times. Consequently the likelihood of
$N$ reduces to
\begin{align*}
  \prod_{t \in [0, 1]} ( \alpha_0(t, X) Y(t) )^{\Delta N(t)} \exp
  \Big(- \int_0^1 \alpha_0(t, X) Y(t) dt \Big)
\end{align*}
where the first product is actually finite, and where we used the fact
that $\prodi_{t \in [0, 1]} (1 - f(t)) = \exp( -\int_0^1 f(t)
dt)$ for a continuous function $f$ on $[0, 1]$. Thus, the
likelihood ratio $L(\alpha, \alpha_0) = d P_\alpha / d P_{\alpha_0}$
writes
\begin{equation*}
  L(\alpha, \alpha_0) = \prod_{t \in [0, 1]} \Big( \frac{\alpha(t,
    X)}{\alpha_0(t, X)} \Big)^{\Delta N(t)} \exp \Big(- \int_0^1
  ( \alpha(t, X) - \alpha_0(t, X) Y(t)) dt \Big),
\end{equation*}
which entails~\eqref{eq:log_likeli0} since $\sum_{t \in [0, 1]}
f(t) \Delta N^i(t) = \int_0^1 f(t)
dN(t)$. Equation~\eqref{eq:log_likeli} is a consequence
of~\eqref{eq:log_likeli0}, together with the fact since $N^1, \ldots,
N^n$ are independent, they cannot jump at the same time, so that
$\sum_{i=1}^n \Delta N^i(t) \in \{ 0, 1\}$ a.s.

{\small

\bibliographystyle{ims}
\bibliography{biblio}

}

\end{document}